\documentclass{amsart}
 
\usepackage{mystyle}
\title{A class of $2$-local finite spectra which admit a $v_2^1$-self-map}
\author[Prasit B.]{Prasit Bhattacharya$^{1,*}$}
\address{$^1$Department of Mathematics, University of Virginia, 131 Kerchof Hall, Charlottesville, VA 22904}
\address{$^1$Tel: +1(434)924-4919}
\address{$^*$Corresponding author}
\email{$^1$pb9wh@virginia.edu}
\author[P. Egger]{Philip Egger$^2$}
\address{$^{2}$Hummel Lab, Campus Biotech, Swiss Federal Institute of Technology Lausanne (EPFL), 1201 Geneva, Switzerland}
\address{$^2$Tel: +41(21)695-5040}
\email{$^2$philip.egger@epfl.ch}

\begin{document}

\maketitle 
\tableofcontents

\begin{abstract}
At the prime $2$, Behrens, Hill, Hopkins and Mahowald showed that $M_2(1,4)$ admits a $32$-periodic $v_2$-self-map. More recently, in joint work with Mahowald, we showed that $A_1$ also admits a $32$-periodic $v_2$-self-map. This leads to the question of whether there exists a finite $2$-local complex with periodicity less than $32$. We answer this question in the affirmative by producing a class of finite $2$-local spectra $\ZZ$ all of which admit a $1$-periodic $v_2$-self-map.
\end{abstract}
\begin{center}Keywords: stable homotopy, $v_2$-periodicity \end{center}
\section{Introduction} \label{Sec:intro} 
Let $\mathcal{C}_0$ be the category of $p$-local finite spectra, where $p$ is a fixed prime. A $v_n$-self-map of an object $X$ of $\mathcal{C}_0$ is a self-map $v: \Sigma^t X \to X$
such that
\[ K(n)_*v : K(n)_*X \to  K(n)_*X \]
is an isomorphism. Here $K(n)$ is the $n$-th Morava $K$-theory and it is well-known that $K(n)_*$ is the graded ring $\F_p[v_n, v_n^{-1}]$ with $|v_n| = 2p^{n}-2$. Since $K(n)_*$ is a graded field (i.e. every nonzero homogeneous element has a multiplicative inverse), $K(n)_*X$ is a graded vector space over $K(n)_*$ and the isomorphism $K(n)_*v$, up to a change of basis, is multiplication by a nonzero element of $K(n)_*$. We say that a $v_n$-self-map $v$ has periodicity $k$ if $k$ is the smallest integer such that $K(n)_*v$ induces multiplication by $v_n^k$ (in which case $t = (2p^{n}-2)k$). We will refer to a $v_n$-self-map of periodicity $k$ as a $v_n^k$-self-map.

In $1998$, Hopkins and Smith showed in \cite{HS} that for every $n\geq 0$, $\mathcal{C}_{n}$, the category of $K(n-1)$-acyclics, is a thick subcategory of $\mathcal{C}_0$ and the $\mathcal{C}_{n}$ form a sequence of thick subcategories 
\[ \mathcal{C}_0 \supset \mathcal{C}_1 \supset \mathcal{C}_2 \supset \dots \supset \mathcal{C}_{\infty}\]
where $\mathcal{C}_{\infty}$ is the category of contractible spectra. A $p$-local finite spectrum $X$ is said to be of type $n$ if $X \in \mathcal{C}_n \setminus \mathcal{C}_{n+1}$. They also showed that 
\begin{thm}[Hopkins-Smith] \label{thm:HS1}Every $p$-local finite spectrum $X$ of type $n$ admits a $v_n$-self-map  
\[v: \Sigma^{k(2p^n-2)}X \to X.\]
Moreover, the cofiber $Cv$ is a spectrum of type $n+1$. 
\end{thm}
Not only does Theorem~\ref{thm:HS1} show the existence of $v_n$-self-maps, but it also provides a recipe for constructing type $n$ spectra. However, \cite{HS} does not shed any light on the minimal periodicity of such a $v_n$-self-map, except to establish that the minimal periodicity is always a power of $p$. 

One of the key properties of a $v_n$-self-map $v:\Sigma^t X \to X$ is that the iterated compositions 
\[ v^{\circ r}:=\Sigma^{(r-1)t} v \circ \dots \circ \Sigma^tv \circ v: \Sigma^{rt}X \to X\]
are homotopically nontrivial and potentially give us an infinite family of elements in the stable homotopy groups of spheres. Typically,  $v^{\circ r}$  composed with the inclusion of a bottom cell
\[ 
\xymatrix{
\tau_r: S^{rt} \ar@{^(->}[r]^{\mathit{incl}} &\Sigma^{rt}X \ar[r]^{ v^{\circ r}}& X 
}
\]
is nontrivial. Therefore, the map $\tau_r$ factors through some skeleton, say $X^{\langle n_r \rangle}$, in such a way that the composite with the pinch map to a top cell of $X^{\langle n_r \rangle}$
\[ 
\xymatrix{
\sigma_r: S^{rt} \ar[r]^{\tau_r} & X^{\langle n_r \rangle} \ar[r]^{\text{pinch}} & S^{n_r} 
}
\]
is a nontrivial element of $\pi_{rt -n_r}(S^0)$. The collection of such $\lbrace \sigma_r: r > 0 \rbrace$ forms an infinite family. The smaller the periodicity of $v$, the smaller the gap in degree between successive elements in the family. Hence the interest is in 
\begin{itemize}
\item finding the minimal periodicity of the $v_n$-self-map on a given type $n$ finite spectrum, and
\item finding finite $p$-local spectra whose $v_n$-self-maps have periodicity as low as possible.
\end{itemize}

Recall that $K(0) = H\mathbb{Q}$ and $v_0$ is just multiplication by $p$. The sphere spectrum $S^0$ is a type $0$ spectrum which admits a $v_0^1$-self-map. Since $S^0$ admits a $v_0^1$-self-map, any type $0$ spectrum admits a $v_0^1$-self-map. 

The search for $v_n$-self-maps gets increasingly complicated as $n$ increases. First, we remind ourselves some of the standard notations used in the literature. The cofiber of the $v_0^{i_0}$-self-map of $S^0$, i.e. multiplication by $p^{i_0}$, is called the $i_0$-th Moore spectrum at the prime $p$ and is denoted by $M_p(i_0)$. By Theorem~\ref{thm:HS1}, $M_p(i_0)$ must admit a $v_1$-self-map of some periodicity and the cofiber of 
\[ v_1^{i_1}:\Sigma^{i_1(2p-2)}M_p(i_0) \to M_p(i_0)\]
is denoted by $M_p(i_0,i_1)$. In general, the cofiber of 
\[ v_n^{i_n}:\Sigma^{i_n(2p^n-2)}M_p(i_0, \dots, i_{n-1}) \to M_p(i_0, \dots, i_{n-1})\]
is denoted by $M_p(i_0,\dots, i_{n-1},i_n)$ and called a generalized Moore spectrum. Often in the literature a generalized Moore spectrum $M_p(i_0,\dots, i_n)$ with $i_k =1$ for $0 \leq k \leq n$ is called a Smith-Toda complex and is denoted by $V_p(n)$. Alternatively, one can define the spectrum $M_p(i_0, \dots, i_n)$ as a topological realization of the $BP_*$-comodule 
\[ BP_*/\langle v_0^{i_0}, v_1^{i_1}, \dots, v_n^{i_n} \rangle. \]

Generalized Moore spectra may not exist for all sequences $(i_0, \ldots, i_n)$, and even if such a spectrum exists, it may not be unique due to the potential non-uniqueness of the self-maps. Toda \cite{Toda} showed that $V_p(1)=M_p(1,1)$ exists for $p \geq 3$, $V_p(2)= M_p(1,1,1)$ exists for $p \geq 5$ and $V_p(3) = M_p(1,1,1,1)$ exists for $p \geq 7$. In $1966$, J.F. Adams proved in \cite{Adams} that $M_2(1)$ does not admit a $v_1^1$-self-map, in fact the minimal periodicity of a $v_1$-self-map on $M_2(1)$ is $4$. Thus, $M_p(1, i)$ does not exist for $i<4$. In $2003$, Behrens and Pemmaraju \cite{BP} showed that $V_3(1) = M_3(1,1)$ admits a $v_2$-self-map of minimal periodicity $9$. Therefore $M_3(1,1,i)$ does not exist for $i<9$. In $2008$, Behrens, Hill, Hopkins and Mahowald \cite{BHHM} showed that the $v_2$-self-map of $M_2(1,4)$ has minimal periodicity $32$. Little is known about $v_n$-self-maps for $n \geq 3$ aside from the work of Toda mentioned above, and Nave's proof in \cite{Nave} of the nonexistence of $V_p(\frac{p+1}{2})$  for $p > 7$.

Instead of focusing on generalized Moore spectra one can also ask the following question: 
\begin{ques}For a fixed prime $p$, what is the type $n$ spectrum whose $v_n$-self-map has the smallest periodicity?\end{ques}
For instance, at the prime $2$, we have seen that $M_2(1)$ does not admit a $v_1^1$-self-map. However, it is known that $Y:=M_2(1) \sma C\eta$ admits eight $v_1^1$-self-maps (see~\cite{DM81}). At the prime $3$, Behrens and Pemmaraju \cite{BP} showed that $M_3(1,1)$ does not admit a $v_2^1$-self-map. However, they also proved that $M_3(1,1) \sma Y(2)$, where $$Y(2) = S^0 \cup_{\alpha_1} S^4 \cup_{2\alpha_1}S^8,$$ admits a $v_2^1$-self-map. Toda proved that $M_p(1,1)$ admits a $v_2$-self-map for $p \geq 5$ and that $M_p(1,1,1)$ admits a $v_3^1$-self-map for $p \geq 7$.

Though $M_2(1,4)$ had a $v_2^{32}$-self-map, the authors hoped that one of the cofibers of the $v_1^1$-self-maps on $Y$, collectively referred to as $A_1$, might have a $v_2$-self-map of periodicity less than $32$. In joint work with Mark Mahowald (see \cite{BEM}), we showed that this does not occur, as all of the spectra $A_1$ have $v_2^{32}$-self-maps. This led to the question of whether there exists any $2$-local type $2$ spectrum with a $v_2$-self-map of periodicity less than $32$. We answer this question in the affirmative by producing a class of finite spectra that admit $v_2^1$-self-maps. The main purpose of this paper is to prove the following theorem. 
\begin{main} \label{main:v2}There is a collection of $2$-local type $2$ spectra $\ZZ$, such that every $Z \in \ZZ$ admits a $v_2$-self-map of periodicity $1$. 
\end{main} 
For the rest of the paper, we will work in the stable homotopy category of $2$-local spectra. Let $A$ denote the mod $2$ Steenrod algebra and $A(n)$ be the subalgebra of $A$ generated by $\lbrace Sq^{2^i}: 1 \leq i \leq n \rbrace$. Let $Q_n$ for $n\geq 0$ be the $n$-th Milnor element in $A$, iteratively constructed using the formula
\[Q_0=Sq^1, Q_n = [Sq^{2^n},Q_{n-1}] = Sq^{2^n}Q_{n-1} - Q_{n-1}Sq^{2^n}.\]
The Milnor element $Q_n$ generates an exterior algebra $E(Q_n)$ as  $Q_n^2 =0$ and it commutes with every $a \in A(n)$. The exterior algebra $E(Q_n)$ is a normal subalgebra of $A(n)$ and the pushout
\[A(n)\modmod E(Q_n) = A(n) \otimes_{E(Q_n)}\Ft = \Ft \otimes_{E(Q_n)} A(n)\]
is an $A(n)$-module, in fact, it is an $A(n)$-algebra. Let $B(n)$ denote the $A(n)$-algebra $A(n)\modmod E(Q_n)$ and 
\[ q_n: A(n) \to B(n) \]
denote the ``quotient'' map.

\begin{defn} \label{defn:Z}
The class $\ZZ$ is the collection of all \emph{finite} spectra $Z$ such that there is an isomorphism of $A(2)$-modules
\[ H^*(Z) \iso B(2).\]
\end{defn}
\begin{rem} It is worth pointing out that the finiteness criterion in Definition~\ref{defn:Z} is essential. Note that for a finite $2$-local spectra $X$, Bousfield localization with respect to $M_2(1)$ is isomorphic to the localization with respect to $H\mathbb{F}_2$ (see \cite{Bou}), i.e. $X = X_{M_2(1)} \simeq X_{H \mathbb{F}_2}$. In all arguments involving the Adams spectral sequence, including the proof of Main Theorem~\ref{main:v2}, we rely on the assumption that every $Z\in\ZZ$ satisfies $Z\simeq Z_{H\mathbb{F}_2}$. If we dropped the finiteness criterion from Definition~\ref{defn:Z}, then one could find spectra $X\in\ZZ$, for example $X=Z \vee K(2) \in \ZZ$, for which $H^*(X)\iso B(2)$, but $X \ncong X_{H\mathbb{F}_2}$. Such an $X$, while an element of $\ZZ$, would not be of type 2, nor would the proof of Main Theorem~\ref{main:v2} be correct in its case.
\end{rem}

\begin{notn}Any $A$-module which restricts to $A(2)$ as an $A(2)$-module will be denoted $A_2$. Likewise, any $A$-module which restricts to $B(2)$ as an $A(2)$-module will be denoted $B_2$.\end{notn}
Definition~\ref{defn:Z} is motivated by the fact that the cohomology of the spectrum $Y$, which admits a $v_1^1$-self-map, is 
\[ H^*(Y) = B(1).\]
To show that the class $\ZZ$ is nonempty, we first enrich the $A(2)$-module $B(2)$ to an $A$-module $B_2$. We do this by enriching the $A(2)$-module $A(2)$ to an $A$-module $A_2$ and taking $B_2$ to be the image of $q_2$. As we will point out in Remark \ref{1600structures} and expand upon in the appendix, there are many different $A$-modules $A_2$, thus there are potentially many different $A$-modules $B_2$. We will then topologically realize the $B_2$ as cohomologies of spectra (see Theorem~\ref{thm:todaR2}).

There is yet another way of obtaining spectra in the class $\ZZ$. It is known \cite[Lemma~$6.1$]{HM} that there exists a nontrivial self-map 
\[ \gamma:\Sigma^5 C\eta \sma C\nu \to C \eta \sma C \nu ,\]
where  $\eta$ and $\nu$ are the well-known Hopf maps in $\pi_{*}(S^0)$. The map $\gamma$ has multiple lifts  
\[ w: \Sigma^{5}A_1 \sma C\nu \to A_1 \sma C\nu\]
whose cofibers $Cw$ belong to the class $\ZZ$. This approach to producing $Z \in \ZZ$ is described in the second author's doctoral thesis \cite[Chapter~$3$]{E}. The authors believe that any spectrum in the class $\ZZ$ can be obtained as a cofiber of such a degree $5$ self-map of $A_1 \sma C\nu.$

Given such a $B_2$, we see that the spectra realizing it are not unique, even up to homotopy. Depending on the specific $B_2$ we choose, there are either $4$ or $8$ different homotopy classes of spectra that realize $B_2$ (see Theorem~\ref{thm:nonunique}). 

Let $k(n)$ denote the connected cover of the $n$-th Morava $K$-theory. Lellmann \cite{Lel} proved that there exists an isomorphism of $A$-modules 
\[ H^{*}(k(n)) \iso A \modmod E(Q_n).\]  
Hopkins and Mahowald \cite{HM} showed that as an $A$-module  
\[ H^*(\mathit{tmf}) \iso A\modmod A(2).\]
 Since every $Z \in \ZZ$ is a realization of $A(2)\modmod E(Q_2)$, we have 
\[ H^{*}(\mathit{tmf} \sma Z) \iso A\modmod A(2) \otimes A(2) \modmod E(Q_2) \iso A \modmod E(Q_2) \iso H^*(k(2)).\]
As a result, any spectrum $Z \in \ZZ$ satisfies the relation
\begin{equation} \label{eqn:Zspecial} 
\mathit{tmf} \sma Z \simeq k(2).
\end{equation}
Thus $Z$ can be thought of as the height $2$ analogue of the spectrum $Y$ because  
\[ \mathit{ko} \sma Y \simeq k(1).\]
As discussed earlier, $Y$ admits a $v_1^1$-self-map. Main~Theorem~\ref{main:v2} produces a $v_2^1$-self-map of $Z$ which further extends the analogy between $Y$ and $Z$.

Understanding the $ko$-resolution of $Y$ (see \cite{M1} and \cite{M2}) results in the proof of the telescope conjecture at chromatic height $1$ at the prime $2$. By analogy, we hope that the $\tmf$-resolution of $Z$ will enable us to attack the telescope conjecture at chromatic height $2$ at the prime $2$. Indeed, the authors have computed a close approximation of the ``easier'' side of the telescope conjecture, namely $ \pi_*(L_{K(2)}Z)$, for any $Z \in \ZZ$. Part of this computation also appears in \cite{E}, with a more detailed report to appear in \cite{BE}.
\subsection*{Organization of the paper} 
In Section~\ref{SEC:Amodule} we show that every $A$-module $B_2$ is ``half'' of a corresponding $A_2$, in that there is a short exact sequence of $A$-modules\[0\to\Sigma^7B_2\to A_2\to B_2\to0.\]

In Section~\ref{SEC:TodaRealize} we recall a criterion of Toda for realizing a given $A$-module as the cohomology of a spectrum, and give a proof of a more refined criterion. In the process we review the construction of Adams towers and dual Adams towers, which will be necessary in Section~\ref{SEC:nonuniqueness}.

In Section~\ref{SEC:RealizeZ} we show that all $A$-modules $B_2$ satisfy Toda's criterion and can thus be topologically realized. However, in Section~\ref{SEC:nonuniqueness}, we show that all of these realizations are non-unique; given an $A$-module $B_2$, there are, up to homotopy, either $4$ or $8$ different spectra that realize it. 

Finally, in Section~\ref{SEC:mainproof} we complete the proof of Main~Theorem~\ref{main:v2}. 

We provide Appendix~\ref{SEC:Appendix1} to show how to obtain $A$-module structures on $B(2)$ in practice. We obtain an explicit $A$-module and display it in the format required by Bruner's Ext program \cite{Bru}. We also display various Ext charts obtained by running this program.

\section*{Acknowlegments}
The authors would like to thank Mark Behrens, Paul Goerss and Mike Mandell for their invaluable assistance and encouragement throughout this project. We would like to thank Irina Bobkova and Nicolas Ricka discussions helpful toward formulating \eqref{cond:A}. We are also indebted to Bob Bruner for his Ext calculator program. While none of the results in this paper rely on computer-assisted proofs, computer-assisted calculations have provided many of the insights in the paper. Finally, we are grateful to Alex Kruckman for making available online an Adem relations calculator, which was very handy for our purposes.

\section{$A$-module structures on $B(2)$} \label{SEC:Amodule}
Let $Q_n$ be the Milnor element of $A$, let $M$ be a left $A(m)$-module for $m\geq n$ and let
\begin{eqnarray*}\mathcal{Q}_n^R:\Sigma^{2^{n+1}-1}M&\to&M\\x&\mapsto&xQ_n\end{eqnarray*}
be the multiplication by $Q_n$ on the right. Adams and Margolis \cite{AM} used the property $Q_n^2 = 0$ to define the Adams-Margolis homology 
\[ H(M; Q_n) = \frac{\ker \mathcal{Q}_n^R}{\img\mathcal{Q}_n^R}.\]
When $M = A(n)$, the right action of $Q_n$ is same as the left action of $Q_n$ as $Q_n$ lies in the center of $A(n)$. Hence we can consider the map $\mathcal{Q}_n$ of multiplication by $Q_n$ on the left or the right. Note that $\coker \mathcal{Q}_n$ is isomorphic to $B(n)$. It can be easily checked that $H(A(n);Q_n) =0$, therefore $\ker \mathcal{Q}_n \iso \img \mathcal{Q}_n$. Moreover, the induced map $\widetilde{\mathcal{Q}}_n$
\[ \xymatrix{
A(n) \ar[drr]^<(.53){\mathcal{Q}_n \ \ \ } \ar@{->>}[r] & \coker \mathcal{Q}_n \ar@{-->}[d]_<(.54){ \ \ \widetilde{\mathcal{Q}}_n}\ar[dr]\\
& \Sigma^{-7}\img \mathcal{Q}_n \ar@{^(->}[r] & \Sigma^{-7}A(n) \\
}\]
is in fact an isomorphism of $A(n)$-modules. As a result (also see \cite{Mit}) we have the short exact sequence of $A(n)$-modules

\[ 0 \to \Sigma^{2^{n+1}-1}B(n) \iso \img\mathcal{Q}_n  \to A(n) \overset{q_n}\to B(n)\iso \coker \mathcal{Q}_n \to 0.\]


Now we restrict our attention to $n=2$ and consider the short exact sequence of $A(2)$-modules
\[ 0 \to \Sigma^{7}B(2) \to A(2) \overset{q_2}\to B(2) \to 0.\]
Since this short exact sequence is not split, it corresponds to a nontrivial element 
\[ \tilde{v}_2 \in Ext^{1,7}_{A(2)}(B(2),B(2)).\]
Let $B_2$ denote an arbitrary left $A$-module whose underlying $A(2)$-module structure is $B(2)$. In Theorem~\ref{thm:E2} we argue that $\tilde{v}_2$ lifts to an element 
\[ \overline{v}_2 \in Ext_{A}^{1,7}(B_2,B_2).\]
Thus there exists a short exact sequence of left $A$-modules 
\begin{equation} \label{eqn:exactA}
 0 \longrightarrow \Sigma^{7}B_2 \overset{i_2}\to A_2 \overset{q_2}\to B_2 \to 0,
\end{equation}
where  the underlying $A(2)$-module structure of the $A$-module $A_2$, is free over one generator in degree $0$.

Before proving Theorem~\ref{thm:E2}, we indulge ourselves in some preliminary computations of certain $Ext$ groups. Let $\Gen$ denote a basis for $B(2)$ as a graded $\Ft$-vector space. Observe that, as an $E(Q_2)$-module 
\[ B(2) \iso \bigoplus_{c \in \Gen} \Sigma^{|c|}\Ft, \]
therefore
\[ DB(2) \otimes B(2) \iso \bigoplus_{c \in D\Gen \times \Gen} \Sigma^{|c|}\Ft,\]
where $D\Gen$ is the basis of $DB(2)$ dual to $\Gen$. Consequently, 
\[ Ext^{*,*}_{E(Q_2)}(B(2),B(2)) \iso \Ft[v_2] \otimes \left(\bigoplus_{c \in D\Gen \times \Gen} \Sigma^{|c|}\Ft\right), \]
where $v_2$ is the image of the periodicity generator of $Ext_{E(Q_2)}^{*,*}(\Ft, \Ft)$ in bidegree $(s,t)=(1,7)$ induced by the unit map 
\[\iota:DB(2) \otimes B(2) \to \Ft.\] 
One can use a change of rings isomorphism to see that 
\[ Ext^{*,*}_{A(2)}(B(2),B(2)) \iso Ext^{*,*}_{E(Q_2)}(\Ft,B(2)) \iso \Ft[v_2] \otimes \left(\bigoplus_{c \in D\Gen} \Sigma^{|c|}\Ft\right).\]
We summarize the above discussion with the following lemma. 
\begin{lem} \label{lem:computations}Let $\Gen$ be a basis for $B(2)$ as a graded $\Ft$-vector space, $D\Gen$ be the corresponding basis for the dual $DB(2)$ and $\iota_0 \in \Gen$ be the unique generator in degree $0$. Then we have isomorphisms      
\begin{enumerate} 
\item \[Ext^{*,*}_{E(Q_2)}(B(2), B(2)) \iso \Ft[v_2] \otimes \left(\bigoplus_{c \in D\Gen \times \Gen} \Sigma^{|c|}\Ft\right),\]
\item \[Ext^{*,*}_{A(2)}( B(2), B(2))  \iso \Ft[v_2] \otimes \left(\bigoplus_{c \in D\Gen} \Sigma^{|c|}\Ft\right),\]
\end{enumerate}
and the inclusion of $E(Q_2)$ into $A(2)$ induces the $v_2$-linear map
\begin{eqnarray*}l:Ext^{s,t}_{A(2)}( B(2), B(2)) &\to& Ext^{s,t}_{E(Q_2)}( B(2), B(2))\\  
\overline{g} &\mapsto&  (\overline{g}, \iota_0), \end{eqnarray*}
for every $\overline{g} \in D\Gen$.
\end{lem}

\begin{thm} \label{thm:E2} Let $B_2$ denote any $A$-module which restricts as an $A(2)$-module to $B(2)$. Then there exists an element  $\overline{v}_2\in Ext_{A}^{1,7}( B_2,B_2)$ which maps to $\tilde{v}_2$ under the map 
\[ k: Ext_{A}^{s,t}(B_2,B_2) \to Ext_{A(2)}^{s,t}(B(2),B(2)). \]
\end{thm} 
\begin{proof}
Consider a minimal free $A$-module resolution of $DB_2 \otimes B_2$, 
\begin{equation} \label{eqn:free}
 \ldots \to F_i \overset{f_{i}}\to \ldots \to F_2 \overset{f_2}\to F_1 \overset{f_1}\to F_0 \overset{f_0}\to DB_2 \otimes B_2.
\end{equation}
A minimal resolution has the property that all the differentials in the sequence 
\[ Hom_{A}(F_0, \Ft) \to Hom_{A}(F_1, \Ft) \to Hom_{A}(F_2, \Ft) \to \ldots \] 
are trivial. Thus, 
\begin{equation} \label{eqn:minidentityA}
 Ext_{A}^{i,*}(B_2,B_2) \iso  Hom_{A}^*(F_i, \Ft) \iso \Ft \langle \text{$A$-module basis of $F_i$}\rangle. 
\end{equation}

The identity map $1_{B_2}:B_2 \to B_2$ generates a nontrivial element
\[ x_{0,0} \in Ext_{A}^{0,0}(DB_2 \otimes B_2, \Ft),\]
which corresponds to a basis element of $F_0$ by \eqref{eqn:minidentityA}.

Now consider a minimal $A(2)$-module resolution of $DB(2) \otimes B(2)$, 
\begin{equation} \label{eqn:freeA2}
 \ldots \to G_i \overset{g_i}\to \ldots \to G_2 \overset{g_2}\to G_1 \overset{g_1}\to G_0 \overset{g_0}\to DB(2) \otimes B(2).
\end{equation}
There is a basis element $y_{0,0} \in G_0$, which corresponds to \[ 1_{B(2)} \in Ext_{A(2)}^{0,0}( B(2), B(2)), \] such that  
$k(x_{0,0}) = y_{0,0}.$ From the $A(2)$-module structure of $B(2)$, it is clear that
\[ g_0(a \cdot y_{0,0}) \neq 0\]
for any $a \in A(2)$ with $0 \leq |a| \leq 6$ and  
\[ g_0(Q_2 \cdot y_{0,0}) = 0.\]
Therefore there is a generator $y_{1,7} \in G_1$ such that 
\[ g_1(y_{1,7}) = Q_2 \cdot y_{0,0},\]
which corresponds to $\tilde{v}_2$. Since any $a \in A$ with $0 \leq |a| \leq 7$ belongs to $A(2)$, the same assertion holds for the element $x_{0,0}$, i.e. there is a basis element $x_{1,7} \in F_1$ such that 
\[ g_1(x_{1,7}) = Q_2 \cdot x_{0,0}.\]
 The generator $x_{1,7}$ will correspond to an element
\[ \overline{v}_2 \in Ext_{A}^{1,7}(DB_2 \otimes B_2, \Ft)\]
with the desired property.
\end{proof}

This shows that any $A$-module structure on $B(2)$ can be obtained as $\img\mathcal{Q}_2^R$ (or $\coker\mathcal{Q}_2^R$) for a given $A$-module $A_2$. In her thesis \cite{Roth}, Marilyn Roth showed that there exist $1600$ different $A$-module structures on $A(2)$. Thus to obtain an $A$-module structure on $B(2)$ in practice, we should consider a left $A$-module structure $A_2$ on $A(2)$, and consider the $A$-modules $\img \mathcal{Q}_2^R$ or $\coker \mathcal{Q}_2^R$. But not every $A$-module structure on $A(2)$ will lead to an $A$-module structure on $B(2)$ as the underlying $A(2)$-module structure of $\img\mathcal{Q}_2^R$ (or $\coker\mathcal{Q}_2^R$) may not be isomorphic to $B(2)$. We do not know of an explicit $A$-module structure on $A(2)$ for which the underlying $A(2)$-module structure on $\img\mathcal{Q}_2^R$ differs from $B(2)$, however, one cannot easily exclude the existence thereof. 
In the following lemma we give a condition which guarantees that the underlying $A(2)$-module structure on $\img \mathcal{Q}_2^R$ and $\coker \mathcal{Q}_2^R$ is precisely $B(2)$. 
\begin{lem} Let $A_2$ denote an $A$-module whose underlying $A(2)$-module structure is simply $A(2)$. Then the underlying $A(2)$-module structure of $\img \mathcal{Q}_2^R$ is $B(2)$ if and only if $A_2$ satisfies 
\begin{equation} \label{cond:A}
Q_2 Sq^8 Q_2 \cdot  \i = 0,
\end{equation}
where $\i$ is the generator in degree $0$.
\end{lem}
\begin{proof}
Let $\mathcal{Q}_2$ denote the map of multiplication by $Q_2$ (on the left or right) in the category of $A(2)$-modules and let $\mathcal{Q}_2^R$ denote the map of multiplication by $Q_2$ on the right in the category of $A$-modules. The underlying $A(2)$-module structure on $\img \mathcal{Q}_2^R$ and $\coker \mathcal{Q}_2^R$ is precisely $B(2)$ unless there exists $n \geq 3$ such that $Sq^{2^n}Q_2 \cdot \i$ does not belong to $\img \mathcal{Q}_2$. 
For dimensional reasons we only need to check the case when $n=3$. Note that $Q_2 Sq^8 Q_2 \cdot \i = 0$ if and only if 
 $$Sq^8Q_2 \cdot \i  = Q_2a \cdot \i = aQ_2 \cdot \i$$ for some $a \in A(2)$  as $\ker \mathcal{Q}_2 \iso \img \mathcal{Q}_2$ and the result follows. 
\end{proof}

Because $|Q_2Sq^8Q_2|=22$ while the highest degree of $A(2)$ is $23$, degree reasons alone are insufficient to guarantee that \eqref{cond:A} is satisfied.

\begin{rem}\label{1600structures}(Number of $A$-module structures on $B_2$) Since any $A$-module structure on $B(2)$ can be produced as a quotient of an $A$-module structure of $A(2)$ which satisfies \eqref{cond:A}, one can in principle count the number of $A$-module structures on $B(2)$ using the results of \cite{Roth}. This method of counting is extremely tedious as there are $1600$ different $A$-module structures on $A(2)$. Moreover, the number of $A$-module structures does not reflect deeper concepts, nor is it directly related to the purpose of this paper. Nonetheless, in the appendix we discuss in detail how to use \cite{Roth} to produce $A$-module structures on $B(2)$ and demonstrate it via an example. The authors would be curious to know if there is a more elegant method of counting $A$-module structures on $B(2)$. 
\end{rem}

\section{Toda's realization theorems} \label{SEC:TodaRealize}
The purpose of this section is to review Toda's criteria for realizing an $A$-module as the cohomology of a spectrum. In the process we will review how to build Adams towers and dual Adams towers, which are essential in Section~\ref{SEC:nonuniqueness}. 

Let $M$ be any graded bounded below $A$-module. Toda \cite[Lemma~$3.1$]{Toda} gave a criterion for the existence of a spectrum $X$ which realizes $M$, i.e., 
\[ H^{*}(X) = M.\]
\begin{thm}[Toda] \label{thm:todaR1}Let $M$ be a graded $A$-module which is bounded below. If for every $n$ such that $M^n \neq 0$, one has  
\[ Ext_{A}^{s,n+s-2}(M, \Ft) = 0 \text{ for every $s \geq 3$,}\]
then there exists a bounded below spectrum $X$ such that $H^*(X) = M$.
\end{thm}
There is yet another realization theorem of \emph{finite} $A$-modules due to Toda. Because we will only ever consider finite $A$-modules, this finiteness hypothesis can be made at no cost to the rest of the paper.
\begin{thm}[Toda] \label{thm:todaR2}Let $M$ be a finite graded $A$-module. If 
\[ Ext_A^{s, s-2}(M, M) = 0 \text{ for every $s \geq 3$,} \]
then there exists a bounded below spectrum $X$ such that $H^*(X) = M$.
\end{thm}
A sketch proof of Theorem~\ref{thm:todaR2} can be found in notes of Haynes Miller~\cite{Miller}. While in a version of Theorem~\ref{thm:todaR2} is proved in \cite[Appendix A]{BKS}, it is significantly more abstract as the theorem is proved in the much more general context of triangulated categories with additional properties. We take this as an opportunity to give a proof of Toda's Realization theorem in its original form, which may be easier for first time readers to follow. We merely complete all the arguments of the sketch proof given in \cite{Miller}. 

First note that Theorem~\ref{thm:todaR2} is stronger than Theorem~\ref{thm:todaR1}, as the realization criterion in Theorem~\ref{thm:todaR1} implies the realization criterion in Theorem~\ref{thm:todaR2} when $M$ is finite. To see this, consider the algebraic Atiyah-Hirzebruch spectral sequence 
\begin{equation} \label{algAHSS}
E_1^{s,t,n} :=  M^n \otimes Ext_{A}^{s,t}(M, \Ft) \Rightarrow Ext_{A}^{s,t-n}(M,M).
\end{equation}
If for every $s \geq 3$ and every $n$ such that $M^n \neq 0$, we have $Ext_{A}^{s,n+s-2}(M, \Ft) = 0$, then it follows that for all $s\geq 3$ and all $n\in\mathbb{Z}$, we have 
\[E_1^{s,n+s-2, n} = M^{n} \otimes Ext_{A}^{s, n+s-2}(M, \Ft) = 0.\]
Thus, the Atiyah-Hirzebruch spectral sequence forces\[Ext_A^{s, s-2}(M,M) = 0\] for every $s\geq 3$.

The broad idea is to consider a free $A$-module resolution of $M$ 
\[ \ldots \overset{d^i}\to F_i \to \ldots \overset{d^1}\to F_1 \overset{d^0}\to F_0 \twoheadrightarrow M\]
and build a corresponding tower of spectra 
\[ \to X_i \to \ldots \to X_1 \to X_0 \]  
often called the \emph{dual Adams tower}, such that 
\[ H^*(\inlim X_i) \iso M.\]

Recall that the \emph{Adams tower} for a spectrum $X$ consists of spectra $\lbrace \tilde{X}_r: r \geq 0 \rbrace$ with maps 
\[ 
\xymatrix{ 
\tilde{X}_0 = X \ar[d]_{\tilde{k}_0} & \ar[l]_{\tilde{i}_0} \tilde{X}_1 \ar[d]_{\tilde{k}_1}& \ar[l]_{\tilde{i}_1} \tilde{X}_2 \ar[d]_{\tilde{k}_2}  & \ar[l]_{\tilde{i}_2} \dots & \ar[l] \tilde{X}_r \ar[d]_{\tilde{k}_r} & \ar[l]_{\tilde{i}_r} \dots \\
\tilde{K}_0 \ar@{-->}[ur]^{\tilde{s}_0} & \tilde{K}_1 \ar@{-->}[ur]^{\tilde{s}_1} & \tilde{K}_2 \ar@{-->}[ur]^{\tilde{s}_2} & \dots & \tilde{K}_r \ar@{-->}[ur]^{\tilde{s}_r} & \dots,
}
\]
such that 
\begin{itemize}
\item $\tilde{K}_r$ is a generalized Eilenberg-Mac Lane spectrum (GEM), 
\item the sequence 
\[ \tilde{X}_{r+1} \overset{\tilde{i}_r}\to \tilde{X}_r \overset{\tilde{k}_r}\to \tilde{K}_r\]
is a cofibration,
\item $\tilde{s}_r:\tilde{K}_r \to \tilde{X}_{r+1}$ is the connecting map of degree $-1$ of the above cofiber sequence, and 
\item the composite $\tilde{d}_r=  \Sigma \tilde{k}_{r+1} \circ \tilde{s}_r: \tilde{K}_r \to \Sigma \tilde{X}_{r+1} \to \Sigma \tilde{K}_{r+1}$ induces the map 
 \[ d^r: F_{r+1} := H^{*}(\Sigma^{r+1}\tilde{K}_{r+1}) \to  F_{r} := H^{*}(\Sigma^r \tilde{K}_{r}) .\]
\end{itemize}
Let $X_r$ be the cofiber in the cofiber sequence 
\[ \tilde{X}_r \to X \to X_r.\]
The map $\tilde{i}_r: \tilde{X}_{r+1} \to \tilde{X}_r$ induces a map $i_r$
\[ 
\xymatrix{
\tilde{X}_{r+1} \ar[d]_{\tilde{i}_r} \ar[r] & X \ar[r] \ar@{=}[d] & X_{r+1}  \ar@{-->}[d]^{i_r}\\
\tilde{X}_r\ar[d]_{\tilde{k}_r} \ar[r] & X \ar[r] \ar[d] & X_r \ar@{-->}[d]^{k_r} \\
\tilde{K}_r \ar[r] & \ast \ar[r] & K_r
}
\]
such that 
\[ X_{r+1} \overset{i_r}\to  X_{r} \overset{k_r}\to  K_{r}\]
forms a cofiber sequence, where $K_{r} \simeq \Sigma \tilde{K}_r$. The collection $\lbrace X_r : r \geq 0 \rbrace$ is the dual Adams tower. Adams showed that if $X$ is a bounded below spectrum, then we have 
\[ \inlim \tilde{X}_i \simeq \ast.\]
Therefore, we also have 
\[ \inlim X_i \simeq X_p,\]
where $X_p$ is the $p$-completion of $X$. 
Just like the Adams tower, the dual Adams tower of a spectrum $X$ fits into the diagram 
\begin{equation} \label{eqn:dualtower}
\xymatrix{ 
X_0 = \ast \ar[d]_{k_0} & \ar[l]_{i_0} X_1 \ar[d]_{k_1}& \ar[l]_{i_1} X_2 \ar[d]_{k_2}  & \ar[l]_{i_2} \dots & \ar[l] X_r \ar[d]_{k_r} & \ar[l]_{i_r} \dots \\
K_0 \ar@{-->}[ur]^{s_0} & K_1 \ar@{-->}[ur]^{s_1} & K_2 \ar@{-->}[ur]^{s_2} & \dots & K_r \ar@{-->}[ur]^{s_r} & \dots,
}
\end{equation}
where $s_i:K_i \to \Sigma X_{i+1}$ are the connecting maps of the fiber sequences
\[ X_{i+1} \to X_i \to K_i\]
and the composite 
\[ k_{i+1} \circ s_i: K_i \to K_{i+1}\]
induces the map $d^r:F_{r+1} \to F_r.$

\begin{proof}[Proof of Theorem~\ref{thm:todaR2}] 
Consider a free $A$-module resolution of $M$ 
 \[ \ldots \overset{d^r}\to F_r \overset{d^{r-1}}\to \ldots \overset{d^1}\to F_1 \overset{d^0}\to F_0 \to M.\]
Let $K_r$ be the GEM such that $H^{*}(K_r) \iso \Sigma^{1-r}F_r$ (note $K_r$ exists as $M$ is finite). We intend to build a dual Adams tower as in \eqref{eqn:dualtower} 
corresponding to this free resolution. Using the condition 
\[ Ext_A^{s,s-2}(M,M) = 0\] 
for $s \geq 3$, we will show that $M$ splits off $H^*(X_r)$ via the maps $p_r$ and $t_r$ as displayed in the diagram 
\begin{equation} \label{eqn:split}
\xymatrix{
  &M \ar[rr]^{\iso} \ar[dr]_{p_{2}} && M \ar[rr]^{\iso} \ar[dr]^{p_3} && M \ar@{}[dr]^{\dots}\ldots \\
H^*(X_{1}) \ar[rr]_{i_{1}^*} \ar[ur]^{t_{1}}&& H^*(X_{2}) \ar[ur]^{t_{2}} \ar[rr]_{i_{2}^*}&& H^*(X_{3})\ar[rr]_{i_{3}^*}\ar[ur]^{t_{3}} && \ldots 
} 
\end{equation}
Let $X = \inlim X_r$. The above splitting will ensure  
\[ \colim H^*( X_r) \iso M. \]
\begin{component}[Case 1: $r=0,1$ and $2$] The first few cases are straightforward. We choose $X_0= \ast$. Since 
\[ X_1 \to X_0 \to K_0\]
is a fiber sequence, it is immediate that $X_1 = \Sigma^{-1}K_0.$ Choose $k_1 = d_0$ and let $X_2$ be the fiber in the sequence
  \[ X_2 \to X_1  \overset{k_1}\to K_1. \]
Now because $ d_1 \circ \Sigma^{-1}d_0: \Sigma^{-2}K_0 \to K_2$ is trivial, we can construct the map $k_2$ in the diagram
\[ 
\xymatrix{
\Sigma^{-1}X_1 \ar[r]^{\Sigma^{-1} k_1} & \Sigma^{-1}K_1 \ar[d]_{d_1} \ar[r] & X_2 \ar[r] \ar@{-->}[dl]^{k_2} & X_1 \\ 
& K_2.
}
\]
Let $X_3$ be the fiber of $k_2$:
\[ X_3 \to X_2  \overset{k_2}\to K_2, \] 
and let $t_1$ be the projection map 
\[ H^*(X_1) \iso H^*(\Sigma^{-1}K_0) \iso F_0 \longrightarrow M.\]
\end{component}
Producing the map $k_3: X_3 \to K_3$ and the splitting $t_2:H^*(X_2) \to M$ is the first nontrivial step of an inductive argument. 
\begin{component}[Case 2: $r=3$]
The fiber sequence $X_2 \to X_1 \to K_1$ produces a long exact sequence 
\[ \ldots \from \Sigma^{-1}F_0 \overset{d^0}\from \Sigma^{-1} F_1 \from H^*(X_2) \from F_0 \overset{d^0}\from F_1 \from \ldots.\] 
Since the cokernel of the map $d^0: F_1 \to F_0$ is $M$, we have the exact sequence of the top row in the diagram
\begin{equation} \label{eqn:split2}
\xymatrix{
0 \ar[r] & M \ar[r]^{p_2} & H^*(X_2) \ar[r]  & \Sigma^{-1}\ker d^0 \ar[r] \ar[d] & 0 \\
\Sigma^{-1} F_4 \ar[r]_{d^3}& \Sigma^{-1} F_3 \ar[r]_{d^2}\ar@{-->}[u]_{c_2} & \Sigma^{-1} F_2 \ar[r]_{d^1}\ar[ur] \ar@{-->}[ul]^{b_2} \ar[u]_{k_2^*}& \ar@{-->}[ul]^{ \ \ \ \ a_2}|\hole \Sigma^{-1} F_1
}
\end{equation}
Since $d^1 \circ d^2 = 0$ and the right vertical arrow of the above diagram is a monomorphism,  the map $k_2^* \circ d^2$ factors through a map $c_2: \Sigma^{-1}F_3 \to M$. Notice that 
\[ p_2 \circ c_2 \circ d^3 = k^2 \circ d^2 \circ d^3 = 0\] 
and $p_2$ is injective, hence \[ c_2 \circ d^3 =0.\]
 Therefore the map $c_2$ represents a class 
\[ \overline{c}_2 \in Ext_A^{3,1}(M, M).\]
Since $Ext_A^{3,1}(M, M) = 0$ by hypothesis, $c_2$ is a coboundary, i.e. $c_2$ factors through $d^2$  
\[ c_2 = b_2 \circ d^2 \]
and 
\[ k_2^* \circ d^2 = p_2 \circ b_2 \circ d^2.\]
So if we replace $k_2^*$ with $k_2^*-p_2\circ b_2$, which is exactly of the type of alteration we are allowed to make, we see that 
\[k_2^* \circ d^2 = 0. \]
 Since the target of the map $k_2$ is a GEM, the algebraic alteration of $k_2^*$ can be realized topologically. Therefore, we have a map $k_3$ in the diagram 
\[ 
\xymatrix{
\Sigma^{-1}X_2 \ar[r]^{k_2} & \Sigma^{-1}K_2 \ar[d]_{d_2} \ar[r] & X_3 \ar[r] \ar@{-->}[dl]^{k_3} & X_2 \\ 
& K_3.
}
\]
Define $X_4$ to be the fiber of the map $k_3$. Since, $k_2^* \circ d^2 = 0$,  $k_2^*$ factors through $d^1$ by a map which we denoted by  $a_3$ in \eqref{eqn:split2}. Consequently, the exact sequence of the top row in the diagram of \eqref{eqn:split2} splits and we have  
\[t_2: H^*(X_2) \to M.\]
\end{component}

\begin{component}[Case 3: $r>3$]
Now inductively assume that we have constructed 
\begin{itemize}
\item $k_{r-1}: X_{r-1} \to K_{r-1}$,
\item $X_r$ as the fiber of the map $k_{r-1}$, and 
\item A diagram of maps 
\begin{equation} \label{eqn:indsplit}
\xymatrix{ 
0 \ar[r]&  M \ar@<.3ex>[r]^{p_{r-1} \ \ \ } & \ar@/^/[l]^{t_{p-1}}H^*(X_{r-1}) \ar[r] & \Sigma^{-r+2} \ker d^{r-3} \ar[r] \ar[d]&  0 \\
&& \Sigma^{-r+2} F_{r-1} \ar[u]^{k_{r-1}^*}\ar[r]_{d^{r-2}}& \Sigma^{-r+2} F_{r-2} \ar[ul]^{a_{r-1}}
}
\end{equation}
whose top row is split exact.
\end{itemize}
The fiber sequence 
\[ X_r \to X_{r-1} \to K_{r-1} \]
produces the horizontal exact sequence in the diagram 
\begin{equation}
\xymatrix{
& 0 \ar[d] &&& 0 \ar[d]\\
& M \ar@<.5ex>[d]&&& \Sigma M \ar@<.5ex>[d] \\
\Sigma^{-r+2}F_{r-1} \ar[d]_{d^{r-2}}\ar[r]^{k_{r-1}^*} &H^*(X_{r-1})\ar@/^/[u]^{t_{p-1}} \ar[d]\ar[r]^{i_{r-1}^*} & H^*(X_r) \ar[r]^{\delta_{r-1}^*}& \Sigma^{-r+1}F_{r-1} \ar[r]^{k_{r-1}^*} \ar[d]_{d^{r-2}} & H^{*+1}(X_{r-1})\ar@/^/[u]^{t_{p-1}} \ar[d] \\
\Sigma^{-r+2}F_{r-2} \ar[ur]^{a_{r-1}}& \ar[l]\Sigma^{-r+2} \ker d^{r-3} \ar[d]&&\Sigma^{-r+1}F_{r-2} \ar[ur]^{a_{r-1}}& \ar[l] \Sigma^{-r+1} \ker d^{r-3} \ar[d] \\
& 0 &&& 0.
}
\end{equation}
The vertical split exact sequences are the part of the assumptions for the inductive step.  By a diagram chase in the above diagram we find that
\begin{itemize}
\item the image of the map \[ k_{r-1}^*: \Sigma^{-r+2}F_{r-1} \to H^*(X_{r-1})\] is $\Sigma^{-r+2} \ker d^{r-3}$, hence
\[\ker i_{r-1}^* \iso \coker k_{r-1}^* \iso M ,\] 
and,
\item the kernel of the map 
\[ k_{r-1}^* : \Sigma^{-r+1}F_{r-1} \to H^{*+1}(X_{r-1}) \]
is isomorphic to $\ker d^{r-2}$, hence 
\[ \img \delta_{r-1}^* \iso \Sigma^{-r+1}\ker d^{r-2}.\]
\end{itemize}
Consequently, the top row in the diagram 
\begin{equation} \label{eqn:splitR}
\xymatrix{
0 \ar[r] & M \ar[r]^{p_r} & H^*(X_r) \ar[r]  & \Sigma^{-r+3}\ker d^{r-2} \ar[r] \ar[d] & 0 \\
\Sigma^{-r+1} F_{r+2} \ar[r]_{d^{r+1}}& \Sigma^{-r+1} F_{r+1} \ar[r]_{d^r}\ar@{-->}[u]_{c_r} & \Sigma^{-r+1} F_r \ar[r]_{d^{r-1}}\ar[ur] \ar@{-->}[ul]^{b_r} \ar[u]_{k^r}& \ar@{-->}[ul]^{ \ \ \ \ \ \ \ \ a_r}|\hole \Sigma^{-r+1} F_{r-1}
}
\end{equation}
is an exact sequence. The diagram in \eqref{eqn:splitR} is just the generalization of the diagram in \eqref{eqn:split2}. Therefore, one can make exactly the same arguments as in the case of $r=3$, to conclude that $k_r$ factors through a map $c_r$. The map $c_r$ is a cocycle and it represents a class  
\[ \overline{c}_r \in Ext_{A}^{r+1,r-1}(M,M).\] 
Since $Ext_A^{r+1,r-1}(M,M) =0$ by hypothesis, $c_r$ is also a coboundary, hence factors through $d^r$ via a map $b_r$. Replacing $k_r^*$ with $k_r^* - p_r \circ b_r$, we see that  \[d_r \circ k_r = 0 \] in the diagram
\[ 
\xymatrix{
\Sigma^{-1}X_r \ar[r]^{k_r} & \Sigma^{-1}K_r \ar[d]_{d_r} \ar[r] & X_{r+1} \ar[r] \ar@{-->}[dl]^{k_{r+1}} & X_{r} \\ 
& K_{r+1}.
}
\]
Hence we have a lift $k_{r+1}: X_{r+1} \to K_{r+1}$. Define $X_{r+2}$ to be the fiber of $k_{r+1}$. Since $k_r^* \circ d^r = 0$, $k_r^*$ factors through $d^{r-1}$ via the map $a_r$. Therefore, the top row of the diagram in Equation~\ref{eqn:splitR} splits and we have a map
\[ t_r: H^*(X_r) \to M.\]
\end{component}
\begin{component}[Convergence] Let $X = \inlim \tilde{X}_r$. We still need to show that 
\[ H^{*}(\inlim \tilde{X}_r) = \underset{\to}\lim H^{*}( \tilde{X}_r) \]
in order to conclude $H^{*}(X)= M$. This is true when $M$ is bounded below. The argument is standard and is known as Adams' Convergence Theorem in the literature (see \cite[Theorem~$2.1$]{Adams2} and \cite[Part III, Theorem~$15.1$]{Adams3} for details). 
\end{component}
\end{proof}

Now we briefly discuss some basic properties of the Adams tower and the dual Adams tower of a spectrum. Let $X$ be a $k$-connected spectrum, i.e. $\pi_{i}(X) =0$ for $i \leq k$. The Adams tower $\set{\tilde{X}_r: r \geq 0}$ and the dual Adams tower $\set{X_r: r \geq 0}$ of $X$ is said to be \emph{minimal} if it corresponds to a minimal free $A$-module resolution of $H^*(X)$. Because minimal free resolutions of $H^*(X)$ are nonunique, it follows that minimal Adams towers and dual Adams towers of $X$ are also nonunique. It follows from the construction that for a minimal Adams tower of $X$, we will have
\begin{equation} \label{eqn:Adamstower}
 Ext_{A}^{s,t}(H^*(\tilde{X}_r), H^*(Y)) \iso \left\lbrace \begin{array}{ccc}
 Ext_{A}^{s+r,t}(H^*(X), H^*(Y))  & \text{for $s \geq k+1$} \\
0 & \text{ for $s \leq k$}
 \end{array} \right.
\end{equation}
for any spectrum $Y$. The above isomorphism is realized by the map 
\[ \tilde{X}_r \to X.\]
Similarly, a minimal dual Adams tower satisfies 
\begin{equation} \label{eqn:Adamstower}
Ext_{A}^{s,t}(H^*(X_r), H^*(Y)) \iso \left\lbrace \begin{array}{ccc}
 Ext_{A}^{s,t}(H^*(X), H^*(Y) )  & \text{for $s \leq r-1 + k$} \\
0 & \text{otherwise}
 \end{array} \right.
\end{equation}
for an arbitrary spectrum $Y$.

Suppose we have a map of spectra $f:X \to Y$, where $X$ and $Y$ are both bounded below. Then $f$ induces a map between their Adams towers 
\[ 
\xymatrix{
\ast \ar[r]\ar@{=}[d] & \ldots \ar[r] & \tilde{X}_3 \ar[r] \ar[d]^{\tilde{f}_3} &\tilde{X}_2 \ar[r] \ar[d]^{\tilde{f}_2} & \tilde{X}_1 \ar[d]^{\tilde{f}_1} \ar[r] & X \ar[d]^{f}\\  
\ast \ar[r] & \ldots \ar[r]& \tilde{Y}_3 \ar[r] &\tilde{Y}_2 \ar[r] & \tilde{Y}_1 \ar[r]& Y  \\
}
\]
and their dual Adams towers
\[ 
\xymatrix{
X \ar[r]\ar[d]_{f} & \ldots \ar[r] & X_3 \ar[r] \ar[d]^{f_3} &X_2 \ar[r] \ar[d]^{f_2} & X_1 \ar[d]^{f_1} \ar[r] & \ast \ar@{=}[d]\\  
Y \ar[r] & \ldots \ar[r]& Y_3 \ar[r] &Y_2 \ar[r] & Y_1 \ar[r] & \ast.\\
}
\]
However, the collection of maps $\set{\tilde{f}_i: i \geq 0}$ and $\set{f_i: i \geq 0}$ may not be unique, even when the Adams tower and its dual are minimal. 



\section{Realization of $B_2$} \label{SEC:RealizeZ}
Let $B_2$ denote a fixed $A$-module structure on the $A(2)$-module $B(2)$. The main purpose of this section is to use Toda's realization theorem, Theorem~\ref{thm:todaR1}, to conclude: 
\begin{thm} \label{thm:existence}There exists a finite spectrum $Z \in \ZZ$ such that 
\[ H^*(Z) \iso B_2 \]
as an $A$-module.
\end{thm}
For this we need to compute $Ext_{A}^{*,*}(B_2, \Ft)$. For any $A$-module $M$ there is a Bousfield-Kan spectral sequence 
\[ E_1^{s,t,n} = \bigoplus_{n} Ext_{A(2)}^{s-n,t}(\overline{A\modmod A(2)}^{\otimes n} \otimes M, \Ft) \Rightarrow Ext_{A}^{s,t}(M, \Ft)\] 
where $\overline{A \modmod A(2)}$ is the augmentation ideal, i.e. the kernel of the map $$A\modmod A(2) \to \Ft.$$ This spectral sequence is also otherwise known as the algebraic-$\tmf$ spectral sequence (see \cite{BHHM, BEM}). We will abbreviate the name to `alg-$\tmf$ SS' for the rest of the paper.
\begin{notn}To save space, we will suppress copies of $\Ft$ in Ext groups, abbreviating $Ext_{A(2)}^{*,*}(N, \Ft)$ to $Ext_{A(2)}^{*,*}(N)$ and $Ext_{E(Q_2)}^{*,*}(K, \Ft)$ to $Ext_{E(Q_2)}^{*,*}(K)$, where $N$ is an $A(2)$-module and $K$ is an $E(Q_2)$-module.  
\end{notn}
\begin{warn} 
The name `algebraic-$\tmf$ spectral sequence' is due to the fact that $H^*(\tmf) = A\modmod A(2)$. However, there are similar spectral sequences involving $A\modmod A(n)$ for $n\geq3$, despite the fact that these $A$-modules are not realizable topologically. We point this out so that readers are aware of the fact that the results in this paper do not rely on the theory of $\tmf$ per se, unlike some other results on $v_2$-self-maps of finite complexes (such as the results of \cite{BEM, BHHM, BP}).
\end{warn}
In \cite{DM82} (also see \cite[\S$5$ ]{BHHM}), it has been proved that as an $A(2)$-module 
\[ A\modmod A(2) \iso \bigoplus_{j \geq 0} \Sigma^{8j} N_1(j) \]
where $N_1(j)$ is the $j$-th Brown-Gitler module \cite{Goe83}. $N_1(0) \iso \Ft$ is precisely the image of the unit map.  As a result we have
\[ \overline{A\modmod A(2)} \iso \bigoplus_{j \geq 1} \Sigma^{8j} N_1(j) \]
and the $E_1$-page of the alg-$\tmf$ SS can be expressed as 
\[E_1^{s,t,n} =\bigoplus_{j_1 \geq 1, \dots, j_n \geq 1} Ext_{A(2)}^{s-n,8(j_1 + \dots +j_n)}(N_1(j_1) \otimes \dots \otimes  N_1(j_n) \otimes M).\]
We will refer to $s$ as the Adams filtration, $t$ as the internal degree and $n$ as the $\tmf$-filtration. Thus, the $d_r$ differentials have tridegree $(1,0,r)$.

In Figure~\ref{fig:tmfSS}, we provide a visual aid to assist the understanding of the $E_1$-page of the alg-$\tmf$ SS. We encode the $\tmf$-filtration using colors and express the spectral sequence in $(x,y) = (t-s,s)$ coordinates. We use black for $n = 0$, blue for  $n = 1$, red for  $n=2$ and green for  $n=3$. We draw the symbol $\arroww$ with a sequence of numbers $j_1 \dots j_k$ at $(t-s,s) = (8(j_1+\dots +j_k) -k, k)$ to indicate that we must place a shifted copy of $Ext_{A(2)}^{*,*}(N_1(j_1) \otimes \dots \otimes N_1(j_k) \otimes M)$ at that bidegree. By doing so, we assemble all the potential contributors to $Ext_{A}^{s,t}(M)$ in the $(t-s,s)$ coordinate system. With this arrangement, where we denote different alg-$\tmf$ filtrations using different colors, any differential in the alg-$\tmf$ SS looks like an Adams $d_1$ differential, pointing one unit up and one unit to the left.   


\begin{figure}[H] 
\begin{sseq}[entrysize=4.7mm, grid = chess]{26}{6}
 \ssdropbull \ssdroplabel[R]{} \ssvoidarrow {0} {1} \ssvoidarrow {2} {1}
 \bl
 \ssmoveto 7 1
 \ssdrop{$1$} \ssdroplabel{} \ssvoidarrow {0} {1} \ssvoidarrow {2} {1} 
 \ssmoveto {15} 1
 \ssdrop{$2$} \ssdroplabel{} \ssvoidarrow {0} {1} \ssvoidarrow {2} {1} 
 \ssmoveto {23} 1
 \ssdrop{$3$} \ssdroplabel{} \ssvoidarrow {0} {1} \ssvoidarrow {2} {1} 
 \red
 \ssmoveto {14} 2
 \ssdrop{$11$} \ssdroplabel{} \ssvoidarrow {0} {1} \ssvoidarrow {2} {1} 
 \ssmoveto {22} 2
 \ssdrop{$12$} \ssdroplabel{} \ssvoidarrow {0} {1} \ssvoidarrow {2} {1} 
 \ssmoveto {22} 2
 \ssdrop{$21$}\ssdroplabel{} \ssvoidarrow {0} {1} \ssvoidarrow {2} {1} 
 
 \green
 \ssmoveto {21} 3
 \ssdrop{$111$} \ssdroplabel{} \ssvoidarrow {0} {1} \ssvoidarrow {2} {1} 
 
 \gr
 \ssmoveto {28} 4
 \ssdropbull \ssdroplabel{} \ssvoidarrow {0} {1} 
 
 \end{sseq} 
\caption{A convenient pictorial description of the $E_1$-page of the alg-$\tmf$ SS} 
\label{fig:tmfSS}
\end{figure}


Now we estimate $Ext_A^{s,t}(B_2)$ using the alg-$\tmf$ SS. Since $B_2$ as an $A(2)$-module is isomorphic to $A(2) \modmod E(Q_2)$, we can apply a change of rings formula to see 
\begin{equation} \label{eqn:CofRB}
Ext_{A(2)}^{*,*}(N_1(j_1) \otimes \dots \otimes N_1(j_k) \otimes B_2) \iso Ext_{E(Q_2)}^{*,*}(N_1(j_1) \otimes \dots \otimes N_1(j_k)). 
\end{equation} 
As an $E(Q_2)$-module 
\[ N_1(1) \iso E(Q_2) \oplus \Sigma^4 \Ft \oplus \Sigma^6 \Ft \]
and 
\[ N_1(2) \iso E(Q_2) \oplus \bigoplus_{2\leq i \leq 4} \Sigma^{2i} E(Q_2) \oplus \bigoplus_{5 \leq i \leq 7} \Sigma^{2i} \Ft.\]

Computation of the $Ext$ groups on the RHS of \eqref{eqn:CofRB} is very tractable. Firstly, $E(Q_2)$ and $\Ft$ are the only indecomposable $E(Q_2)$-modules, which means that any $E(Q_2)$-module $M$ can be expressed as direct sums of shifted copies of $E(Q_2)$ and $\Ft$. Moreover, the fact that  
\begin{itemize}
\item $\Ft \otimes \Ft \iso \Ft$, 
\item $\Ft \otimes E(Q_2) \iso E(Q_2) \otimes \Ft \iso E(Q_2)$, and,
\item $E(Q_2) \otimes E(Q_2) \iso E(Q_2) \oplus \Sigma^7 E(Q_2)$,
\end{itemize}
allows us to express the tensor product $M \otimes N$ of two $E(Q_2)$-modules as a direct sum of indecomposable $E(Q_2)$-modules. Once we know the  indecomposable components of an $E(Q_2)$-module $M$, we can compute $Ext_{E(Q_2)}^{*,*}(M, \Ft)$ using the facts
\begin{itemize}
\item $Ext_{E(Q_2)}^{*,*}(E(Q_2)) \iso \Ft$
and 
\item $Ext_{E(Q_2)}^{*,*}(\Ft) \iso \Ft[v_2]$, where $v_2$ has bidegree $(s,t) = (1,7)$.
\end{itemize}

\begin{figure}[H] 
 \begin{sseq}[entrysize=4.7mm, grid = chess]{26}{6}
 \ssdropbull \vtwo \vtwo \vtwo \vtwo
 
 \bl
 \ssmoveto 7 1 
 \ssdrop{\circ} \ssname{s} \ssdroplabel[U]{h_3}
 \ssmoveto {11} 1  
 \ssdropbull  \ssdroplabel[D]{h_{2,2}}\vtwo \vtwo
 \ssmoveto {13} 1 
 \ssdropbull \vtwo \vtwo
 \ssmoveto {15} 1 
 \ssdrop{\circ}
 \ssmoveto {19} 1 
 \ssdrop{\circ}
 \ssmoveto {21} 1 
 \ssdrop{\circ}
 \ssmoveto {23} 1 
 \ssdrop{\circ}
 \ssmoveto {25} 1 
 \ssdropbull
 
 \red 
 \ssmoveto {14} 2
 \ssdrop{\circ} \ssname{s} \ssdroplabel[U]{h_3^2}
 \ssmoveto {21} 2
 \ssdrop{\circ} 
 \ssmoveto {22} 2
 \ssdropbull \ssname{h}\ssdrop{\circ} \ssdrop{\circ}
 
 \green
 \ssmoveto {21} 3 
 \ssdrop{\circ} \ssname{s} \ssdroplabel[U]{h_3^3}
 \black
 \ssgoto h \ssgoto s \ssstroke[arrowto]
 
\end{sseq} 
\caption{ The $E_1$-page of the alg-$\tmf$ SS for $B_2$} 
\label{fig:tmfSSB}
\end{figure}

In Figure~\ref{fig:tmfSSB} we plot the first $25$ stems of the $E_1$-page of the alg-$\tmf$ SS, annotating certain elements with their May names. Note that the $E_1$-page is a module over $Ext_{E(Q_2)}^{*,*}(\Ft) = \Ft[v_2],$ hence admits a $v_2$-action. We use a $\circ$ to denote an element on which $v_2$ acts trivially, otherwise we use a $\bullet$. We use dotted lines to indicate the $v_2$-action. We will indicate the $\tmf$ filtration
using the same color code as Figure~\ref{fig:tmfSS}.  
\begin{proof}[Proof of Theorem~\ref{thm:existence}] $B_2$ is nonzero in dimensions $0$ through $16$. By Theorem~\ref{thm:todaR1}, it is enough to prove 
\[Ext_A^{s, s+n-2}(M) =0\]
for $s \geq 3$ for $0 \leq n \leq 16$, equivalently, 
\[ Ext_A^{s, t}(M) =0\]
for $s \geq 3$ and $-2 \leq t-s \leq 14$. From, Figure~\ref{fig:tmfSSB} it is clear that in the $E_1$-page of alg-$\tmf$ SS  $E_1^{s,t,n} =0$ for $s \geq 3$ and $-2 \leq t-s \leq 14$ for all $n \in \mathbb{N}$, and hence the result follows.  
\end{proof}
While we have proved what we set out to prove in this section, we would like to justify the May names and the differential in stem $22$ of Figure~\ref{fig:tmfSSB}, as it will be crucial for the proof of our main theorem in Section \ref{SEC:mainproof}. To do this, we must delve deeper into the alg-$\tmf$ SS, identifying the elements by the names they would inherit from the May spectral sequence. We first recall the May filtration of the Steenrod algebra. 

The May filtration, introduced by J.P. May \cite{May}, can be easily described as a decreasing filtration of the dual Steenrod algebra $A_*$ \cite{Rav}, \cite{Ko}.  The \emph{May weight} $w$ of $\xi_i^{2^j}$ is $2i-1$. In general 
\[ w(\xi_{i_1}^{j_1} \dots \xi_{i_n}^{j_n}) = \sum_{k=1}^n (2i_k -1) \alpha(j_k)\]
where  $\alpha(j_k)$ is the number of $1$'s in the $2$-adic expansion of $j_k$.  The associated graded of $A_*$ is a Hopf algebra, which is primitively generated by $\xi_{i,j}$, the image of $\xi_i^{2^j}$ in the associated graded. Consequently we have a filtration of the cobar complex $C(\Ft, A_*, \Ft)$, resulting in a spectral sequence 
\[ E_1^{s,t, w} = \Ft[h_{i,j}: i \geq 1, j \geq 0] \Rightarrow Ext_{A_*}^{s,t}(\Ft, \Ft) \iso Ext_{A}^{s,t}(\Ft, \Ft)  \]
where $h_{i,j} = [\xi_{i,j}]$. This spectral sequence is called the May spectral sequence in the literature.  
\begin{notn}For any $A$-module $M$ we will denote the $E_r$-page of the alg-$\tmf$ SS by $E_r^{s,t,n}[M]$.\end{notn}

Note that we have already established that (see Theorem~\ref{thm:E2}) there is a map 
\[ \overline{v}_2:  \Sigma^{-1,0} B_2 \to \Sigma^{0,7} B_2\]
in the derived category of $A$-modules, whose cofiber is $A_2$, an $A$-module whose underlying $A(2)$-structure is a free copy of $A(2)$.
Note that 
\[ E_1^{s,\ast, n}[A_2] = Ext_{A(2)}^{s-n,\ast}(A(2) \otimes \overline{A \modmod A(2)}^{\otimes n} ) \iso Ext_{\Ft}^{s-n,\ast}( \overline{A \modmod A(2)}^{\otimes n} ) \]
which is isomorphic to 
\[ \overline{(A \modmod A(2))_*}^{\otimes n}\]
when $s = n$  and $0$ otherwise. Consequently, the $E_1$-page of the alg-$\tmf$ SS for $A_2$ (restricted to the part $s=n$) is isomorphic to a sub-complex of $\overline{C^{\ast, \ast}(\Ft, A_*, \Ft)}$ (see Remark~\ref{mayZ}) which we denote by  
\[ \overline{C^{\ast, \ast}(\Ft, (A\modmod A(2))_*, \Ft)}. \]
Moreover, for degree reasons, the alg-$\tmf$ SS collapses at the $E_2$-page, i.e.\linebreak$E_2^{s,t, s}[A_2] = Ext_{A}^{s,t}(A_2)$.  Recall that $(A \modmod A(2))_* = \Ft[\zeta_1^8, \zeta_2^4, \zeta_3^2, \zeta_4, \zeta_5, \dots ]$, where $\zeta_i$ are the anti-automorphic images of $\xi_i$. $(A \modmod A(2))_*$ inherits the May filtration from $A_*$, likewise $\overline{C^{\ast, \ast}(\Ft, (A\modmod A(2))_*, \Ft)}$ inherits the May filtration from $\overline{C^{\ast, \ast}(\Ft, A_*, \Ft)}$.  Thus we have a May spectral sequence calculating the $d_1$-differential of the alg-$\tmf$ SS for $A_2$
\[ E_1^{*,*,*, *} = \Ft[h_{i,j}: i+j \geq 4, i \geq 1, j \geq 0] \Rightarrow E_2^{*,*,*}[A_2] \iso Ext_{A}^{*,*}(A_2). \]
Therefore we can assign May names to elements of $Ext_{A}^{*,*}(A_2)$.  From \eqref{eqn:CofRB} it is clear that every element $x$ in the $E_1$-page of the alg-$\tmf$ SS for $B_2$ satisfies either $v_2^i \cdot x \neq 0$ for all $i>0$ or $v_2^1 \cdot x = 0$. Therefore the map 
\[\overline{q}_2: E_1^{s,t,n}[B_2]  \to E_1^{s,t,n}[A_2],\] 
induced by $q_2:A_2 \to B_2$, is injective when restricted to the subcomplex $E_1^{s,*,s}[B_2]$. Therefore we can assign May names to the those elements in $E_1^{s,t,n}[B_2]$ for which $s = n$. 

The $A_*$-comodule 
\[ \Sigma^{8}N_1(1)_* \subset (A\modmod A(2))_* = \Ft[\zeta_1^8, \zeta_2^4, \zeta_3^2, \zeta_4, \zeta_5, \dots]\]
consists of the elements $\set{\zeta_1^8, \zeta_2^4, \zeta_3^2, \zeta_4}$ (see \cite{DM82, BHHM} for details). The red bullet 
\[ {\color{red}{\bullet}}\in E_1^{2,2+22,2}[B_2] \]
is the contribution of a generator in $Ext_{E(Q_2)}^{0,24}(\Sigma^8N_1(1)\otimes\Sigma^8N_1(1))$ which corresponds to the cobar element $[\zeta_2^4|\zeta_2^4]$. Therefore its  May name is $h_{2,2}^2$. Similarly, the green circle 
\[ {\color{green}{\circ}}\in E_1^{3,3+21,3}\] 
is the contribution of a generator in $Ext_{E(Q_2)}^{3-3,24}(\Sigma^8N_1(1)\otimes\Sigma^8N_1(1)\otimes\Sigma^8N_1(1))$ which corresponds to 
$[\zeta_1^8|\zeta_1^8|\zeta_1^8]$, hence has the May name $h_{3}^3$. By work of Tangora \cite{Tan}, which is also exposed in \cite{Ko}, one has 
 \begin{equation} \label{eqn:diffsigma3}
  d_2(h_{2,2}^2) = h_3^3
 \end{equation}
in the May spectral sequence, which explains the differential in Figure~\ref{fig:tmfSSB}.
\begin{rem} \label{mayZ} Note that $A(2)$ is not a normal subalgebra of the 
Steenrod algebra $A$, so $(A \modmod A(2))_*$ is not a Hopf algebra. Therefore, one \emph{cannot} make sense of  a cobar construction $C^{\ast, \ast}(\Ft,(A \modmod A(2))_*, \Ft )$ in the conventional sense.  However, using the fact that $\tmf \sma A_2 \simeq H\Ft$, we see that there is a map from the $\tmf$-resolution of $A_2$ 
\[  
\xymatrix{
A_2 \ar[d] & \ar[l]  \overline{\tmf} \sma A_2 \ar[d]  & \ar[l] \dots   \overline{\tmf}^{\sma n} \sma A_2 \ar[d] & \ar[l] \dots  \\
\tmf \sma A_2 &  \tmf \sma \overline{\tmf} \sma A_2&  \tmf \sma \overline{\tmf}^{\sma n} \sma A_2
}
\]
to the Adams resolution for the sphere spectrum $S^0$ 
\[  
\xymatrix{
S^0 \ar[d] & \ar[l]  \overline{H\mathbb{F}_2} \ar[d]  & \ar[l] \dots  \overline{H\mathbb{F}_2}^{\sma n} \ar[d] & \ar[l] \dots  \\
H\mathbb{F}_2 &  H\mathbb{F}_2 \sma \overline{H\mathbb{F}_2} &  H\mathbb{F}_2 \sma \overline{H\mathbb{F}_2}^{\sma n}. 
}
\]
Consequently we have an injective map from the $E_1$-page of the $\tmf$-based Adams spectral sequence of $A_2$ to the reduced cobar complex for $S^0$
\[ E_1^{n, \ast}[A_2] = \pi_*(\tmf \sma \overline{\tmf}^{\sma n} \sma A_2) \iso (\overline{A\modmod A(2)})_*^{\otimes n} \hookrightarrow \overline{A}_*^{\otimes n} = \overline{C^{n, \ast}(\Ft, A_*, \Ft)} \]
which commutes with the differentials. Thus $ E_1^{\ast, \ast}[A_2]$ is isomorphic to a subcomplex of $\overline{C^{\ast, \ast}(\Ft, A_*, \Ft)}$, which we denote by $\overline{C^{\ast, \ast}(\Ft,(A \modmod A(2))_*, \Ft )}$. The same can be concluded for $E_1$-page of the alg-$\tmf$ SS for $A_2$ because it is isomorphic to the $E_1$-page of the $\tmf$-based Adams spectral sequence for $A_2$.  
\end{rem}

\section{Nonuniqueness of $Z$} \label{SEC:nonuniqueness}
In this section we prove that the spectra $Z \in \ZZ$ realizing a given $A$-module $B_2$ are never unique, even up to homotopy. In fact, a given $A$-module $B_2$ can be realized by either 4 or 8 homotopically different spectra in $\ZZ$. 
To motivate the proof, we first give a sufficient condition for a spectrum $X$ under which the $A$-module structure of $H^*(X)$ determines $X$ uniquely up to homotopy. 
\begin{prop} \label{prop:unique}Let $X$ be any finite spectrum, and let $M$ denote the $A$-module $H^*(X)$. If 
\[ Ext_A^{s,s-1}(M,M) = 0\]
for every $s \geq 2$, then every $H\mathbb{F}_2$-nilpotently complete spectrum $Y$ such that $H^*(Y) \iso M$ as an $A$-module, is weakly equivalent to $X$.
\end{prop}
\begin{proof} Consider the Adams spectral sequence 
\[E_2^{s,t} = Ext_{A}^{s,t}(H^{*}(Y), H^{*}(X)) \Rightarrow [X,Y]_{t-s}.\]
Since $H^*(X) \iso H^*(Y) \iso M$, the $E_2$ page is isomorphic to $Ext_A^{s,s-1}(M, M)$. Let $g$ denote a generator in bidegree $(0,0)$ which corresponds to the isomorphism $H^*(X) \iso H^*(Y)$ of $A$-modules. Since $Ext_{A}^{r,r-1}(M, M) = 0$ for $r \geq 2$, it follows that 
\[ d_r(g)  = 0. \] 
Thus $g$ realizes a topological map 
\[\overline{g}: X \to Y\]
which induces an isomorphism in cohomology. Therefore, by Whitehead's theorem $X$ is weakly equivalent to $Y$.
\end{proof}
Therefore, to address the uniqueness question we must compute $Ext_{A}^{r,r-1}(B_2,B_2)$ for $r \geq 2$. We will use the alg-$\tmf$ SS to do so. As an $A(2)$-module  $B_2 \iso A(2) \modmod E(Q_2)$, therefore  a change of rings formula implies 
\[ Ext_{A(2)}^{*,*}(N_1(j_1) \otimes \dots \otimes N_1(j_k) \otimes DB_2 \otimes B_2) \iso Ext_{E(Q_2)}^{*,*}(N_1(j_1) \otimes \dots \otimes N_1(j_k) \otimes DB_2).\]
Since
\[ DB(2) \iso \bigoplus_{c \in D\G} \Sigma^{|c|}\Ft\]
 as an $E(Q_2)$-module, we have 
\[ 
\xymatrix{
E_1^{s,t,n}[DB_2 \otimes B_2]  \ar[d]^{\iso}\\
\underset{ \ \ \ j_1 \geq 0, \dots, j_n \geq 0}\bigoplus  Ext_{E(Q_2)}^{s-n,8(j_1 + \dots +j_n)}(N_1(j_1) \otimes \dots \otimes  N_1(j_n) \otimes DB_2)  \ar[d]^{\iso} \\ 
\underset{ \ \ j_1 \geq 0, \dots, j_n \geq 0}\bigoplus  Ext_{E(Q_2)}^{s-n,8(j_1 + \dots +j_n)}(N_1(j_1) \otimes \dots \otimes  N_1(j_n) \otimes \underset{c \in D\G} \bigoplus \Sigma^{|c|}\Ft)  \ar[d]^{\iso} \\ 
\underset{c \in D\G}\bigoplus (\underset{ \ \ \ j_1 \geq 0, \dots, j_n \geq 0}\bigoplus  Ext_{E(Q_2)}^{s-n,8(j_1 + \dots +j_n)+|c|}(N_1(j_1) \otimes \dots \otimes  N_1(j_n))) \ar[d]^{\iso} \\ 
\underset{c \in D\G}\bigoplus  E_1^{s,t +|c|,n}[B_2].
}
\]
In other words,  $E_1^{*,*,*}[DB_2 \otimes B_2]$  is a  direct sum of shifted copies of $E_1^{*,*,*}[B_2]$, one for each generator of $DB_2$. We computed $E_1^{*,*,*}[B_2]$ in Section~\ref{SEC:RealizeZ} and displayed it in Figure~\ref{fig:tmfSSB}. 
\begin{figure}[H] 
  \[ \begin{sseq}[entrysize=10mm, grid = chess]{-2...2}{4}
  \ssmoveto {-2} {0}
  \ssdrop{\ast} 
  \ssmoveto {-1} {0}
  \ssdrop{\ast}
  \ssmoveto {0} {0}
  \ssdropbull\ssdroplabel[R]{\overline{\iota}} \ssname{g}
  \ssmoveto{-2} 1
  \ssdrop{\ast} 
  \ssmoveto{-1} 1
  \ssdrop{\ast} 
  \ssmoveto{0} 1
  \ssdrop{\ast} 
  \ssmoveto{1} 1
  \ssdrop{\ast} 
  \ssmoveto{2} 1
  \ssdrop{\ast} 
  \ssmoveto {-2} 2 
  \ssdrop{\ast}
  \ssmoveto {-1} 2 
  \red \ssdrop{\circ} \black \ssdropbull \ssdropbull
  \ssmoveto {0} 2
  \ssdrop{\ast}
  \ssmoveto {1} 2
  \ssdrop{\ast}
  \ssmoveto {2} 2
  \ssdrop{\ast}
  \ssmoveto {3} 2
  \ssdrop{\ast}
  \ssmoveto{-1} 2
  \ssdrop{} \ssname{a}
 \ssgoto g \ssgoto a \ssstroke[arrowto,dotted]
 \end{sseq} \]
\caption{ The $E_1$-page of the alg-$\tmf$ SS for $DB_2 \otimes B_2$} 
\label{fig:tmfSSBB}
\end{figure}

\begin{notn}  We know that $H_*(Z) \iso \Ft[\xi_1,\xi_2]/ (\xi_1^8, \xi_2^4 )$ as an $A(2)_*$-comodule. 
Let $g_{\xi_1^{i_1}\xi_2^{i_2}}$ be the element in $H^*(Z)$ dual to $\xi_1^{i_1}\xi_2^{i_2}$. We conveniently choose 
\[ \Gen = \set{g_{\xi_1^{i_1}\xi_2^{i_2}}: 0 \leq i_1 \leq 7, 0 \leq i_2 \leq 3}.\]
We denote the element of $D\Gen$, which is Spanier-Whitehead dual to $g_{\xi_1^{i_1}\xi_2^{i_2}}$, by  $\overline{g}_{\xi_1^{i_1}\xi_2^{i_2}}$.
\end{notn}
In Figure~\ref{fig:tmfSSBB} we draw the alg-$\tmf$ SS for $DB_2 \otimes B_2$ for $0 \leq s \leq 3$ and $-2 \leq t-s \leq 2$. We place a $\ast$ in bidegree $(s,t)$ if $E_1^{s,t,*}[DB_2 \otimes B_2] \neq 0$ but irrelevant to this discussion.  In bidegree $(s,t) = (2, 2-1)$, the two bullets ($\bullet$) are $v_2^2\cdot\overline{g}_{\xi_1^7\xi_2^2}$ and $v_2^2\cdot \overline{g}_{\xi_1^4\xi_2^3}$, and the circle (\textcolor{red}{$\circ$}) is $h_3^2 \cdot \overline{g}_{\xi_1^6\xi_2^3}$.
\begin{lem}  \label{lem:4or8}
Let $B_2$ denote any $A$-module structure on $B(2)$. Then  
\[ \dim_{\Ft}(Ext_A^{2,1}(B_2 , B_2)) = \text{$2$ or $3$}\]
depending on the $A$-module structure on $B_2$.
\end{lem}
\begin{proof}
The elements $\set{v_2^2 \cdot \overline{g}_{\xi_1^7\xi_2^2}$, $v_2^2 \cdot \overline{g}_{\xi_1^4\xi_2^3}, h_3^2 \cdot \overline{g}_{\xi_1^6\xi_2^3}}$ cannot support a nontrivial differential in the alg-$tmf$ SS  as $$E_1^{s,s-2,*}[DB_2 \otimes B_2] = 0$$ for all $s = 3$. Moreover, $v_2^2 \cdot \overline{g}_{\xi_1^7\xi_2^2}$ and $v_2^2 \cdot \overline{g}_{\xi_1^4\xi_2^3}$ cannot be a target of  
 a differential in the alg-$\tmf$ SS, as they are in algebraic-$tmf$ filtration zero. Therefore, $v_2^2 \cdot\overline{g}_{\xi_1^7\xi_2^2}$ and $v_2^2 \cdot \overline{g}_{\xi_1^4\xi_2^3}$ are present in $Ext_{A}^{2,1}(B_2, B_2)$.  However, $h_3^2 \cdot \overline{g}_{\xi_1^6\xi_2^3}$ can be a target of a differential in the alg-$\tmf$ SS for $B_2$. 
\end{proof}
\begin{rem} \label{rem:fourrealizations}In fact, for the specific $A$-module structure on $B(2)$ which is worked out in Appendix~\ref{SEC:Appendix1} (see Figure~\ref{ZDZ_1}), we see that 
\[ \dim_{\Ft}(Ext_A^{2,1}(B_2 , B_2)) = 2,\] 
which means that $h_3^2 \cdot \overline{g}_{\xi_1^6\xi_2^3}$ is trivial in $Ext_{A}^{2,1}(B_2, B_2)$. Though the authors are not aware of  an $A$-module structure on $B(2)$ for which $h_3^2 \cdot \overline{g}_{\xi_1^6\xi_2^3}$ is nontrivial in $Ext_{A}^{2,1}(B_2, B_2)$, one cannot rule out such a scenario as we have not been able to get a handle on all possible $A$-module structures on $B(2)$.
\end{rem}
\begin{thm} \label{thm:nonunique} For an $A$-module $B_2$ whose underlying $A(2)$-module is isomorphic to $B(2)$, let 
\[ n= \dim_{\Ft}(Ext_A^{2,1}(B_2 , B_2)).\]
Then there are $2^n$ different homotopy types of spectra $Z \in \ZZ$ such that 
\[ H^*(Z) \iso B_2\]
as an $A$-module.
\end{thm}
\begin{proof}
By Theorem~\ref{thm:todaR2}, we know that there exists at least one spectrum $Z$ that realizes $B_2$. Fix a minimal dual Adams tower $\set{Z_i: i \geq 0}$ for $Z$. The dual Adams tower for $Z$ fits into the diagram 
\[ 
\xymatrix{ 
Z_0 = \ast \ar[d]_{k_0} & \ar[l]_{i_0} Z_1 \ar[d]_{k_1}& \ar[l]_{i_1} Z_2 \ar[d]_{k_2}  & \ar[l]_{i_2} \dots & \ar[l] Z_r \ar[d]_{k_r} & \ar[l]_{i_r} \dots \\
K_0 \ar@{-->}[ur]^{s_0} & K_1 \ar@{-->}[ur]^{s_1} & K_2 \ar@{-->}[ur]^{s_2} & \dots & K_r \ar@{-->}[ur]^{s_r} & \dots
}
\]  
where $K_r$ is a GEM. The spectrum $Z$ is then the limit
 \[ Z = \inlim Z_i.\]

To create another spectrum $Y$ realizing $B_2$, we alter $k_2$ using a nonzero element $\overline{\delta} \in Ext_{A}^{2,1}(B_2, B_2)$ in such a way that  the composite 
  \[ d_r: K_r \overset{s_r} \to \Sigma Z_{r+1} \overset{k_{r+1}} \to \Sigma K_{r+1} \]
remains fixed for all $r \geq 0$. Then we argue that $Y$ and $Z$ are not weakly equivalent. 
 
For an element $\overline{\delta} \in Ext_{A}^{2,1}(B_2,B_2)$, choose a cocycle representative  
\[ \delta^*: \Sigma F_2^* \to B_2.\]
Since $\delta^*$ is a cocycle, the composite 
\[ \Sigma F_3^* \overset{d^2} \to \Sigma F_2^*  \overset{\delta^*} \to M\]
is trivial. Note that $Z_1 = \Sigma^{-1} K_0$ and $k_1 = d_0$, therefore unwinding the long exact sequence associated to the fiber sequence
\[ Z_2 \overset{i_1} \to Z_1 \overset{k_1} \to K_1\]
gives us the diagram
\begin{equation} \label{eqn:changek2}
\xymatrix{
&\Sigma F_2^*\ar[dr]_{\delta^*} \ar@{-->}[r]^-{\Delta^*}&F_0^* \iso H^{*}(Z_1) \ar[dr]^{i_1^*}\ar@{->>}[d]& \\
&0 \ar[r] & M \ar[r]_{p_2} & H^*(Z_2) \ar[dr]_{k_1^*}\ar[r] & \Sigma^{-1}\ker d^0\ar[d] \ar[r] & 0 \\
&&&& H^{*+1}(K_1),
}
\end{equation}
where the horizontal row is  exact (compare with~\eqref{eqn:split2}). Let $\Delta^*$ denote a lift of $\delta^*$, which exists as $\Sigma F_2^*$ is a free $A$-module. 

Now we build a dual Adams tower $\set{Y_i: i \geq 0}$ which fits into the diagram 
\[ 
\xymatrix{ 
Y_0 = \ast \ar[d]_{l_0} & \ar[l]_{j_0} Y_1 \ar[d]_{l_1}& \ar[l]_{j_1} Y_2 \ar[d]_{l_2}  & \ar[l]_{j_2} \dots & \ar[l] Y_r \ar[d]_{l_r} & \ar[l]_{j_r} \dots \\
K_0 \ar@{-->}[ur]^{t_0} & K_1 \ar@{-->}[ur]^{t_1} & K_2 \ar@{-->}[ur]^{t_2} & \dots & K_r \ar@{-->}[ur]^{t_r} & \dots
}
\]  
as follows. Define $Y_i := Z_i$ for $0 \leq i \leq 2$ and  let $l_i = k_i$ for $0 \leq i \leq 1$. Let
\[ l_2:Y_2 \to K_2\] 
 be the map that classifies 
\[ \tilde{k}_2^* + p_2 \circ \delta^*:\Sigma F_2^* \to H^{*}(\tilde{Z}_2).\]
The condition that $ \delta^*$  is a cocycle  guarantees that $d_3 \circ \tilde{l}_2 \simeq 0$. By construction of $l_2$, we have 
\begin{equation} \label{diag:nocommute}  
\xymatrix{
 Y_2 = Z_2 \ar[d]_{l_2 -k_2}\ar[r]^{j_1 =i_1} & Y_1 = Z_1 \ar[dl]_{\Delta}\\
 K_2.
}
\end{equation}
Note that $\Delta \not\simeq 0$ as $\overline{\delta} \neq 0$. However, $(l_2 -k_2) \circ t_1 = \Delta \circ j_1 \circ t_1 = 0$ which means  
\[ l_2 \circ t_1 = k_2 \circ s_1 = d_1.\]
Build the rest of the tower $\{Y_i\}$ in the usual way described in the proof of Theorem~\ref{thm:todaR2}, and let $Y$ denote the limit 
\[ Y := \inlim Y_i.\]
The choices for $l_i$ for $i \geq 3$ will not make any difference as 
\[ Ext_{A}^{s,s-1}(B_2, B_2) = 0\]
for $s \geq 3$.
 
Note that 
\[ Hom_{A}(H^{*}(Y), H^*(Z)) \iso Ext_{A}^{0,0}(H^{*}(Y), H^{*}(Z)) = Ext_{A}^{0,0}(B_2, B_2) = \Ft. \]
Therefore there is exactly one $A$-module map $\iota_*:H^{*}(Y) \to H^*(Z)$ which is an isomorphism. Let $\overline{\iota}$ denote the corresponding element in  $Ext_{A}^{0,0}(H^{*}(Y), H^{*}(Z))$. To show that $Y$ and $Z$ are not weakly equivalent it suffices to show that $\overline{\iota}$ is not a permanent cycle in the Adams spectral sequence 
\begin{equation} \label{eqn:ASS}
E_2^{s,t} = Ext_{A}^{s,t}(H^{*}(Y), H^{*}(Z)) \Rightarrow [Z,Y]_{t-s}.
\end{equation}
Indeed, assume $\overline{\iota}$ is a permanent cycle. Then there will be a map of spectra
\[ \iota: Z \to Y\]
which induces $\iota_*$ on cohomology and we will have a map of dual Adams towers
\[ 
\xymatrix{
Z \ar[r]\ar[d]_{\iota} & \ldots \ar[r] & Z_3 \ar[r] \ar@{-->}[d]^{\iota_3} &Z_2 \ar[r] \ar@{-->}[d]^{\iota_2} & Z_1 \ar@{-->}[d]^{\iota_1} \\  
Y \ar[r] & \ldots \ar[r]& Y_3 \ar[r] &Y_2 \ar[r] & Y_1 \\
}
\]
Note that $\iota_1$ and $\iota_2$ exist trivially as $Z_1 = Y_1$ and $Z_2 = Y_2$. However, the diagram 
 \[ \xymatrix{
Z_3 \ar@{-->}[d]_{\iota_3} \ar[r] & Z_2 \ar@{=}[d] \ar@{}[dr] |{\mathstrut\raisebox{0ex}{ \ \ \textcolor{red}{$\circlearrowright$}\kern1.4em}} \ar[r]^{k_2} & K_2 \ar@{=}[d] \\
Y_3 \ar[r] & Y_2 \ar[r]_{l_2} & K_2
}
\]
shows that the map $\iota_3: Z_3 \to Y_3$ exists if and  only if the right square commutes because the horizontal rows are cofiber sequences. But~\eqref{diag:nocommute} implies that the right square does not commute. Therefore $\iota_3$ does not exist and we have a $d_2$ differential 
\[ d_2(\overline{\iota}) = \overline{\delta}\]
in the Adams spectral sequence \eqref{eqn:ASS}, hence a contradiction.

Thus we get exactly one homotopy type realizing $B_2$ for each element in\linebreak$Ext_A^{2,1}(B_2,B_2)$ and the result follows.    
\end{proof}
\begin{rem} \label{rem:difference} Let $Z_1$ and $Z_2$ denote two spectra such that 
\[ H^*(Z_1) \iso H^*(Z_2) \iso B_2.\]
From the arguments in the proof of Theorem~\ref{thm:nonunique}, we see that in the Adams spectral sequence 
\[ E_2^{s,t} = Ext_{A}^{s,t}(H^{*}(Z_2), H^*(Z_1)) \Rightarrow [Z_1,Z_2]_{t-s}\]
the generator $\overline{\iota}$ in bidegree $(0,0)$ supports a nontrivial $d_2$ differential unless $Z_1 \simeq Z_2$. Therefore the $d_2$ differential in some sense `measures the difference' in homotopy types between the spectra $Z_1$ and $Z_2$. 

\end{rem}

\section{Existence of a $v_2^1$-self-map of $Z$} \label{SEC:mainproof}
As usual, let $B_2$ denote an $A$-module whose underlying $A(2)$-module structure is $B(2)$. By Theorem~\ref{thm:nonunique} and Lemma~\ref{lem:4or8}, we know that there are either four or eight different homotopy types of spectra which realize $B_2$. Let $Z$ denote a spectrum of a specific homotopy type which realizes $B_2$. In Figure~\ref{fig:tmfSSBBB} we lay out the $E_1$-page of the alg-$\tmf$ SS for $DZ \sma Z$. 

\begin{figure}[H] 
  \[ \begin{sseq}[entrysize=8mm, grid = chess]{-2...6}{5}
  \ssmoveto {5} 3
  \green \ssdrop{\circ} \ssname{a} \black
  \ssmoveto {-2} {0}
  \ssdrop{\ast} 
  \ssmoveto {-1} {0}
  \ssdrop{\ast}
  \ssmoveto{-2} 1
  \ssdrop{\ast} 
  \ssmoveto{-1} 1
  \ssdrop{\ast} 
  \ssmoveto{0} 1
  \ssdrop{\ast} 
  \ssmoveto{1} 1
  \ssdrop{\ast} 
  \ssmoveto{2} 1 \ssdrop{\ast} \ssmoveto{3} 1 \ssdrop{\ast} \ssmoveto{4} 1 \ssdrop{\ast} \ssmoveto{5} 1 \ssdrop{\ast}
  \ssmoveto{6} 1  \ssdrop{\ast} 
  \ssmoveto {0} {0}
  \ssdropbull\ssdroplabel[R]{\overline{\iota}}  \vtwo \ssdroplabel[D]{\overline{v}_2}
  \ssmoveto {-2} 2 
  \ssdrop{\ast}
  \ssmoveto {-1} 2 
  \red \ssdrop{\circ} \black \ssdropbull \vtwo 
  \ssmoveto {-1} 2
   \ssdropbull \vtwo 
  \ssmoveto {0} 2
  \ssdrop{\ast}
  \ssmoveto {1} 2
  \ssdrop{\ast}
  \ssmoveto {2} 2
  \ssdrop{\ast}
  \ssmoveto {3} 2
  \ssdrop{\ast}
  \ssmoveto {4} 2
  \ssdrop{\ast}
  \ssmoveto {5} 2
  \ssdrop{\ast}
  \ssmoveto {6} 2
  \red \ssdropbull \ssname{g} \black
  \ssdrop{\ast} 
  
  \ssmoveto {2} 3
  \ssdrop{\ast}
  \ssmoveto {3} 3
  \ssdrop{\ast}
  \ssmoveto {4} 3
  \ssdrop{\ast}
  \ssmoveto {6} 3
  \ssdrop{\ast}

  \ssgoto g \ssgoto a \ssstroke[arrowto]
  \end{sseq} \]
\caption{ The $E_1$-page of the alg-$\tmf$ SS for $DB_2 \otimes B_2$} 
\label{fig:tmfSSBBB}
\end{figure}


As in Figure~\ref{fig:tmfSS} we use colors to distinguish $\tmf$-filtration, with $\ast$ in bidegrees which are nonzero but irrelevant to the discussion. In Theorem~\ref{thm:E2}, we have established that $v_2 \cdot \overline{\iota}$ is a nonzero permanent cycle in the alg-$\tmf$ SS denoted by $\overline{v}_2 \in Ext_{A}^{1,7}(H^*(Z), H^*(Z))$. 
\begin{proof}[Proof of Main~Theorem~\ref{main:v2}]
From Figure~\ref{fig:tmfSSBBB} it is clear that 
\[ Ext_{A}^{s,s+5}(H^{*}(Z), H^*(Z)) =0 \]
for all $s \geq 4$. The task of proving the Main~Theorem~\ref{main:v2} then boils down to eliminating the possibility of $\overline{v}_2$ supporting a nontrivial $d_2$ differential in the Adams spectral sequence 
\begin{equation} \label{eqn:AdamsSS}
 E_2^{s,t} = Ext_{A}^{s,t}(H^{*}(Z), H^*(Z)) \Rightarrow [Z,Z]_{t-s}.
\end{equation}
 Figure~\ref{fig:tmfSSBBB} makes it clear that in the $E_1$ page of alg-$\tmf$ SS for $DZ \sma Z$
 \[ E_1^{3,8,*} = \set{v_2^3 \cdot \overline{g}_{\xi_1^7\xi_2^2}, v_2^3 \cdot \overline{g}_{\xi_1^4 \xi_2^3}, h_3^3 \cdot \overline{g}_{\xi_1^7 \xi_2^3} }.\] 
The element $h_3^3 \cdot \overline{g}_{\xi_1^7 \xi_2^4}$, which has $\tmf$-filtration $3$, is the target of a differential in the alg-$\tmf$ SS due to \eqref{eqn:diffsigma3}, hence trivial in the $E_{2}$-page of \eqref{eqn:AdamsSS}. However, the elements $v_2^3 \cdot \overline{g}_{\xi_1^7\xi_2^2}$ and $v_2^3 \cdot \overline{g}_{\xi_1^4 \xi_2^3}$ are nonzero permanent cycles in the alg-$\tmf$ SS as they are in $\tmf$-filtration $0$, hence represent nonzero elements in the $E_2$-page of \eqref{eqn:AdamsSS}. 

Since $H^*(k(2)) = A\modmod E(Q_2)$, the $E_2$-page of the Adams spectral sequence computing $k(2)_*(DZ \sma Z)$ is
\[ Ext_{E(Q_2)}^{*,*}(H^*(Z), H^*(Z)) \iso Ext_{E(Q_2)}^{*,*}(B(2), B(2))\]
which we have already computed (see Lemma~\ref{lem:computations}). Now consider the map of spectral sequences 
\begin{equation} \label{eqn:mapofSS}
\xymatrix{
Ext_{A}^{s,t}(H^{*}(Z), H^*(Z)) \ar[d]_{k_{\ast}}  \ar@{=>}[r] & \pi_{t-s}(DZ \sma Z) \ar[d]^{k}\\
Ext_{A(2)}^{s,t}(H^{*}(Z), H^*(Z)) \ar[d]_{l_{\ast}}  \ar@{=>}[r] & \tmf_{t-s}(DZ \sma Z) \ar[d]^{l} \\
Ext_{E(Q_2)}^{s,t}(H^*(Z), H^*(Z)) \ar@{=>}[r] & k(2)_{t-s}(DZ \sma Z)
}
\end{equation}
induced by the maps $ S^0 \overset{k}\to \tmf \overset{l}\to k(2)$. The elements $v_2^i \cdot \overline{\iota}$, $v_2^i \cdot \overline{g}_{\xi_1^7\xi_2^2}$ and $v_2^i \cdot \overline{g}_{\xi_1^4, \xi_2^3}$ for $i \geq 0$ have nonzero image  under the map $k_*$ as they are in $\tmf$-filtration zero. Moreover, they have nonzero image under the composite $(l \circ k)_*$ by Lemma~\ref{lem:computations}. 
From Remark~\ref{rem:difference}, it is clear that 
\[ d_2(\overline{\iota}) = 0\]
 in the upper-most spectral sequence in \eqref{eqn:mapofSS}. Therefore,  
\[ d_2 (k_{*}(l_{*}(\overline{\iota})) = 0.\]
Since $k(2)$ admits a $v_2$-self-map, the differentials in the bottom-most spectral sequence of \eqref{eqn:mapofSS} are $v_2$-linear, which means  
\[ d_2 (v_2 \cdot k_{*}(l_{*}(\overline{\iota})) = 0.\]

Now, suppose that 
\[ d_2(\overline{v}_2) = c_1 v_2^3 \cdot \overline{g}_{\xi_1^7\xi_2^2} + c_2 v_2^3 \cdot \overline{g}_{\xi_1^4 \xi_2^3} \]
with $(c_1,c_2) \neq (0,0)$. This would imply
\[ d_{2}(v_2 \cdot k_{*}(l_{*}(\overline{\iota})) \neq 0\] 
as $v_2^3 \cdot \overline{g}_{\xi_1^7\xi_2^2}$ and $v_2^3 \cdot \overline{g}_{\xi_1^4 \xi_2^3}$ have nontrivial images under the map $k_* \circ l_*$, which is a contradiction. Thus we can conclude 
 \[ d_2(\overline{v}_2) = 0\]
and $Z$ admits a $v_2^1$-self-map.
\end{proof}
\begin{rem}[$v_2$-maps]
Now let $Z_1$ and $Z_2$ be spectra realizing $B_2$. Then 
\begin{equation} \label{eqn:diffunit}
 d_2(\overline{\iota}) = c_1 v_2^2 \cdot \overline{g}_{\xi_1^7\xi_2^2}+ c_2 v_2^2 \cdot \overline{g}_{\xi_1^4\xi_2^3} +c_3 h_3^2 \cdot \overline{g}_{\xi_1^6 \xi_2^3}  
\end{equation}
where $c_i \in \Ft$. Essentially using the argument in the proof of Main~Theorem~\ref{main:v2}, we see that there are three possibilities:
\begin{enumerate}
 \item if $c_1=c_2=c_3=0$, then $Z_1$ and $Z_2$ have the same homotopy type, which we call $Z$ and there exists a $v_2^1$-self-map\[v:\Sigma^6Z\to Z,\]
 \item if $c_1=c_2=0$ but $c_3\neq0$, then $Z_1$ and $Z_2$ have different homotopy types, but we nonetheless have a map\[v:\Sigma^6Z_1\to Z_2\]which induces multiplication by $v_2^1$ on Morava $K$-theory, and
 \item if one of $c_1$ or $c_2$ is nonzero, then there is no map\[v:\Sigma^6Z_1\to Z_2\]which induces multiplication by $v_2^1$ on Morava $K$-theory.
\end{enumerate}
 
\end{rem}

\newpage
\appendix
\section{An explicit $A$-module structure on $Z$} \label{SEC:Appendix1}
The $A(2)$-module $B(2)$ is $32$-dimensional as an $\Ft$-vector space spread across from degree $0$ to  degree $16$. 
Endowing $B(2)$ with an $A$-module structure is therefore not easy in practice and requires a systematic approach. In Section~\ref{SEC:Amodule}, we established that any $A$-module structure on $B(2)$ extends to an $A$-module structure on $A(2)$ (see Theorem~\ref{thm:E2}). Hence, we can obtain all possible $A$-module structures on $B(2)$ from $A$-module structures on $A(2)$. Marilyn Roth \cite{Roth} showed that there are $1600$ $A$-module structures on $A(2)$. However, she did not list explicit $A$-module structures or $A_*$-comodule structures. Rather, she encoded these structures in $\mathbb{F}_2$-linear maps
\[ s: A(2)_* \to (A \modmod A(2))_*\]
which satisfy certain criteria. Each $s$-map leads to a unique $A$-module structure on $A(2)$, and Roth showed that there are $1600$ such maps. The purpose of this Appendix is to review Roth's work and demonstrate, via an example, how to obtain different $A$-module structures on $B(2)$ in practice. We will express the $A$-module in the format required by Bruner's Ext software \cite{Bru}  and display the output of the program. 

We begin by describing the recipe for converting an $s$-map into an $A$-module structure on $A(2)$ as prescribed in \cite{Roth}. It is useful to dualize things, considering not $A$-modules, but $A_*$-comodules. Recall that the dual Steenrod algebra $A_*$ is the polynomial algebra
\[A_*=\Ft[\xi_i:i\geq1] \]
where the generator $\xi_i$ is in degree $2^i-1$. As $A(2)\subset A$ is a sub-Hopf algebra generated by $Sq^1$, $Sq^2$ and $Sq^4$, the dual $A(2)_*$ is the quotient algebra
\[ A(2)_* := \Ft[\xi_1, \xi_2, \xi_3]/ ( \xi_1^{8}, \xi_2^4, \xi_3^2 ) \]
of $A_*$, and 
\[ (A \modmod A(2))_* = \Ft[\xi_1^8, \xi_2^4, \xi_3^2, \xi_4, \xi_5, \dots ].\]
Notice that as a graded $\Ft$-vector space, $A(2)_*$ has generators in degrees $0 \leq t \leq 23$. Between degrees $0$ and $23$, the graded $\Ft$-vector space $(A \modmod A(2))_*$ has exactly one nonzero generator in degrees $8$, $12$, $14$, $15$, $20$, $22$ and $23$. In those same degrees, $A(2)_*$ has $3$, $4$, $4$, $3$, $2$, $1$ and $1$ generators respectively. As a result there are $2^{18}$ different $s$-maps. An $s$-map can be uniquely extended to an $\overline{s}$-map
\[ \overline{s}: A_* \to (A\modmod A(2))_*\]
as $A_* = A(2)_* \otimes (A \modmod A(2))_*$. Corresponding to each $s$-map, there is a $j$-map
\[ j: A(2)_* \to A_* \]
 which is defined as follows. Let $\pi: A_* \to A(2)_*$ be the canonical projection and $i:(A\modmod A(2))_* \to A_*$ be the canonical inclusion. Define the $j$-map inductively using the formula 
\begin{equation} \label{eqn:jmap}
 j(a) = a + s(a) + \sum_{i} j(\pi(a_i')) \cdot \overline{s}(i(a_i''))
\end{equation}
where  $a_i'$ and $a_i''$ are part of the formula for the coproduct on $a$ 
\[\psi(a) = a \otimes 1 + 1 \otimes a + \sum_i a_i' \otimes a_i''.\]
 Roth proved that \cite[Chapter III]{Roth} 
\[A(2)_*\overset{j}{\to}A_*\overset{\psi}{\to}A_*\otimes A_*\overset{A_*\otimes\pi}{\to}A_*\otimes A(2)_*\]
defines an $A_*$-comodule structure on $A(2)_*$ if and only if the composite 
\begin{equation} \label{eqn:rothcriteria}
 A(2)_* \overset{j} \to A_* \overset{\psi}\to A_* \otimes A_* \overset{\overline{s} \otimes \overline{s}} \to (A\modmod A(2))_* \otimes (A\modmod A(2))_* \end{equation}
sends $\xi_1^7 \xi_2^3 \xi_3 \mapsto 0$. Roth listed all the $s$-maps which satisfy the criteria \eqref{eqn:rothcriteria}, and found that there are $1600$ such $s$-maps. 

One can obtain the $A$-module structure on $A(2)$ simply by dualizing the $A_*$-comodule structure on $A(2)_*$, however, in practice, this is quite tedious. Instead, one can obtain the $A$-module structure on $A(2)$ directly from the $j$-map using the right action of the total squaring operation 
\[ Sq = \sum_{i \geq 0} Sq^i: A_* \to A_*.\] 
Note that for a left $A$-module $M$, its dual $\widehat{M}= Hom_{\Ft}(M, \Ft)$ admits a right action of $A$. In particular, the right action of $A$ on its dual $A_*$ is determined by the right action of the total squaring operation $Sq$ on $A_*$. It is well-known that the the right action of $Sq$ is a ring homomorphism determined by the formula 
\[ (\xi_i)Sq = \xi_i + \xi_{i-1}.\]
Given a $j$-map we obtain the $A$-module structure on $A(2)$ as follows. We declare \[ Sq^i(x) = x_1 + \dots + x_n \]
where $x, x_1, \dots, x_n \in A(2)$, if $(j(x_{i*}))Sq^i$ contains $j(x_*)$ as a summand. After obtaining the $A$-module structure on $A(2)$ one must check if it satisfies~\eqref{cond:A}. If so, one can easily get an $A$-module structure on $B(2)$ by considering the inclusion map $i_2$ or the quotient map $q_2$ of \eqref{eqn:exactA}.
\subsection{$A$-module definition file for Bruner's program}
We now consider the sample $j$-map that Roth computed \cite[Pg 30, Chapter III]{Roth} and check that the resulting $A$-module structure on $A(2)$ satisfies \eqref{cond:A}. We will now express the resulting $A$-module structure on $B(2)$ as a definition file for Bruner's program \cite{Bru}.  The $A$-module structure is encoded in a text file named \texttt{Z} in a way that we will now describe. The first line of the text file \texttt{Z} consists of a positive integer $n$, the dimension of $B(2)$ as an $\Ft$-vector space, whose basis elements we will call $g_0,\ldots ,g_{n-1}$. The second line consists of an ordered list of integers $d_0,\ldots,d_{n-1}$, which are the respective degrees of the $g_i$. Every subsequent line in the text file describes a nontrivial action of some $Sq^k$ on some generator $g_i$. For instance, if we have\[Sq^k(g_i)=g_{j1}+\cdots +g_{jl},\]we would encode this fact by writing the line\[\texttt{i k l j1 \ldots jl}\]followed by a new line. Every action not encoded by such a line is assumed to be trivial. The text file \texttt{Z} is as follows.
\begin{figure}[H] 
\begin{small}
\begin{verbatim}
32

0 1 2 3 3 4 4 5 5 6 6 6 7 7 7 8 8 9 9 9 10 10 10 11 11 12 12 13 13 14
15 16

0 1 1 1
0 2 1 2
0 3 1 3
0 4 1 5
0 5 1 7
0 6 1 9
0 7 1 12
0 10 1 20
0 12 1 25
0 13 1 27
0 14 1 29

1 2 2 3 4
1 3 1 6
1 4 2 7 8
1 5 1 10
1 6 2 12 13
1 7 1 15
1 8 1 17
1 9 1 20
1 12 1 27
1 14 1 30
1 15 1 31

2 1 1 3
2 2 1 6
2 4 3 9 10 11
2 5 2 12 14
2 6 2 15 16
2 7 1 18
2 8 1 21
2 9 1 23
2 10 1 25
2 11 1 27
2 12 1 29


3 2 1 8
3 3 1 10
3 4 2 12 14
3 6 3 17 18 19
3 7 2 20 22
3 8 2 23 24
3 9 1 26
3 10 2 27 28
3 11 1 29
3 12 1 30
3 13 1 31
\end{verbatim}
\end{small}
\end{figure}
\clearpage
\begin{figure}[H] 
\begin{subfigure}{.33\textwidth}
\begin{small}
\begin{verbatim}
4 1 1 6
4 2 1 8
4 3 1 10
4 4 1 13
4 5 1 15
4 6 1 17
4 7 1 20
4 12 1 30
4 13 1 31

5 1 1 7
5 2 2 9 10
5 3 1 12
5 4 1 16
5 5 1 18
5 6 2 20 22
5 12 1 31

6 2 1 10
6 4 1 15
6 6 1 20
6 12 1 31

7 2 2 12 13
7 3 1 15
7 4 2 17 18
7 5 1 20
7 6 1 24
7 7 1 26

8 1 1 10
8 4 2 17 19
8 5 2 20 22
8 6 1 24
8 7 1 26
8 8 1 28
8 9 1 29
8 10 1 30
8 11 1 31

9 1 1 12
9 2 1 15
9 4 2 20 21
9 5 1 23
9 6 2 25 26
9 7 1 27
9 8 1 29

10 4 2 20 22
10 6 1 26
10 8 1 29
\end{verbatim}
\end{small}
\end{subfigure}%
\begin{subfigure}{.34\textwidth}
\begin{small}
\begin{verbatim}
10 10 1 31

11 1 1 14
11 2 1 16
11 3 1 18
11 4 1 21
11 5 1 23
11 6 1 25
11 7 1 27
11 10 1 31

12 2 1 17
12 3 1 20
12 4 1 23
12 6 2 27 28
12 7 1 29
12 8 1 30
12 9 1 31

13 1 1 15
13 2 1 17
13 3 1 20
13 4 1 24
13 5 1 26

14 2 2 18 19
14 3 1 22
14 4 2 23 24
14 5 1 26
14 6 2 27 28
14 7 1 29
14 8 1 30
14 9 1 31

15 2 1 20
15 4 1 26

16 1 1 18
16 2 1 22
16 4 2 25 26
16 5 1 27
16 6 1 29

17 1 1 20
17 4 1 28
17 5 1 29
17 6 1 30
17 7 1 31

18 2 1 24
18 3 1 26
\end{verbatim}
\end{small}
\end{subfigure}%
\begin{subfigure}{.33\textwidth}
\begin{small}
\begin{verbatim}
18 4 1 27
18 6 1 30
18 7 1 31

19 1 1 22
19 2 1 24
19 3 1 26
19 4 1 28
19 5 1 29
19 6 1 30
19 7 1 31

20 4 1 29
20 6 1 31

21 1 1 23
21 2 2 25 26
21 3 1 27
21 6 1 31

22 2 1 26
22 4 1 29
22 6 1 31

23 2 2 27 28
23 3 1 29
23 4 1 30
23 5 1 31

24 1 1 26
24 4 1 30
24 5 1 31

25 1 1 27
25 2 1 29
25 4 1 31

26 4 1 31

27 2 1 30
27 3 1 31

28 1 1 29
28 2 1 30
28 3 1 31

29 2 1 31

30 1 1 31


\end{verbatim}
\end{small}
\end{subfigure}
\end{figure}
\subsection{Ext charts produced by Bruner's program}
Using Bruner's program,  we compute $Ext_{A}^{*,*}(B_2, \Ft)$ (see Figure~\ref{Z_1}) where $B_2$ is the $A$-module structure on $B(2)$ that we computed above. 
\begin{figure}[h]
 \centering
 \includegraphics[page=1,width=\textwidth]{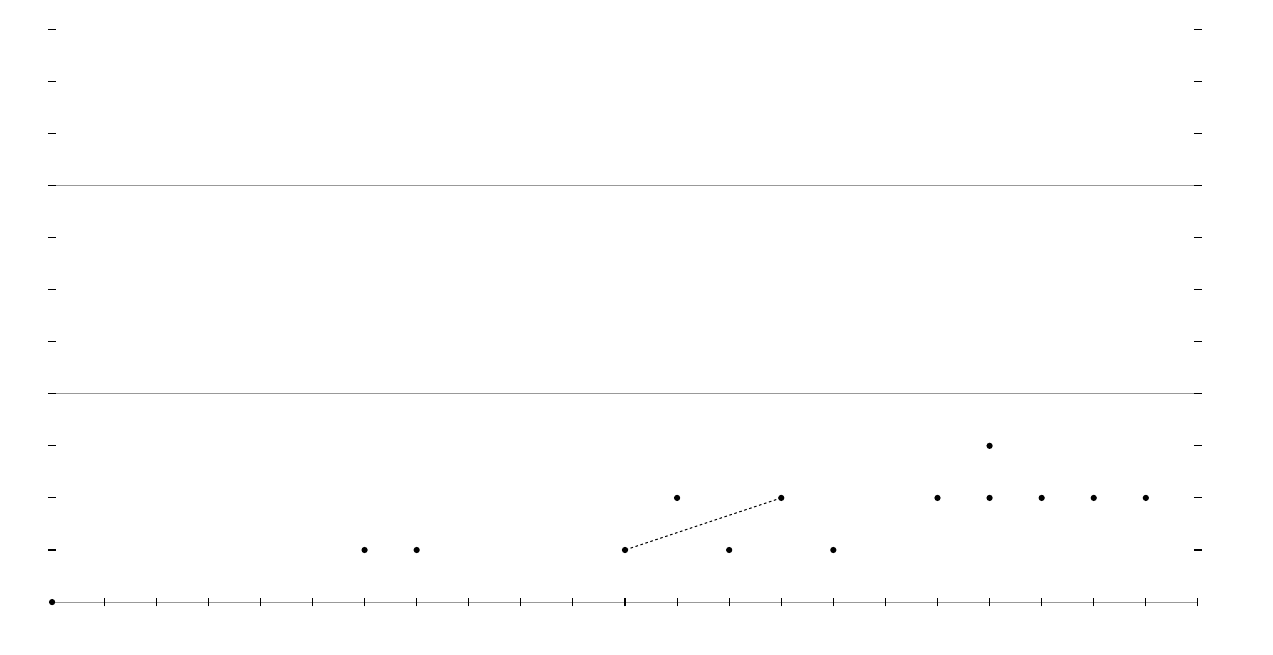}
 \caption{The file \texttt{Z\_1.pdf} displaying $Ext_{A}^{s,t}(B_2, \Ft)$ for \\ $0 \leq t-s \leq 22.$}
\label{Z_1}
\end{figure}
Bruner's program is able to compute the $A$-module structure for Spanier-Whitehead duals and tensor products of $A$-modules. Using these features we produce an $A$-module definition file for $B_2 \otimes DB_2$ which we named \texttt{ZDZ}. This allows us to compute $Ext_{A}^{*,*}(B_2,B_2)$ (see Figure~\ref{ZDZ_1}) easily using Bruner's program.  We use the chart in Figure~\ref{ZDZ_1} to make our conclusions in Remark~\ref{rem:fourrealizations}.
\begin{figure}[h]
 \centering
 \includegraphics[page=1,width=\textwidth]{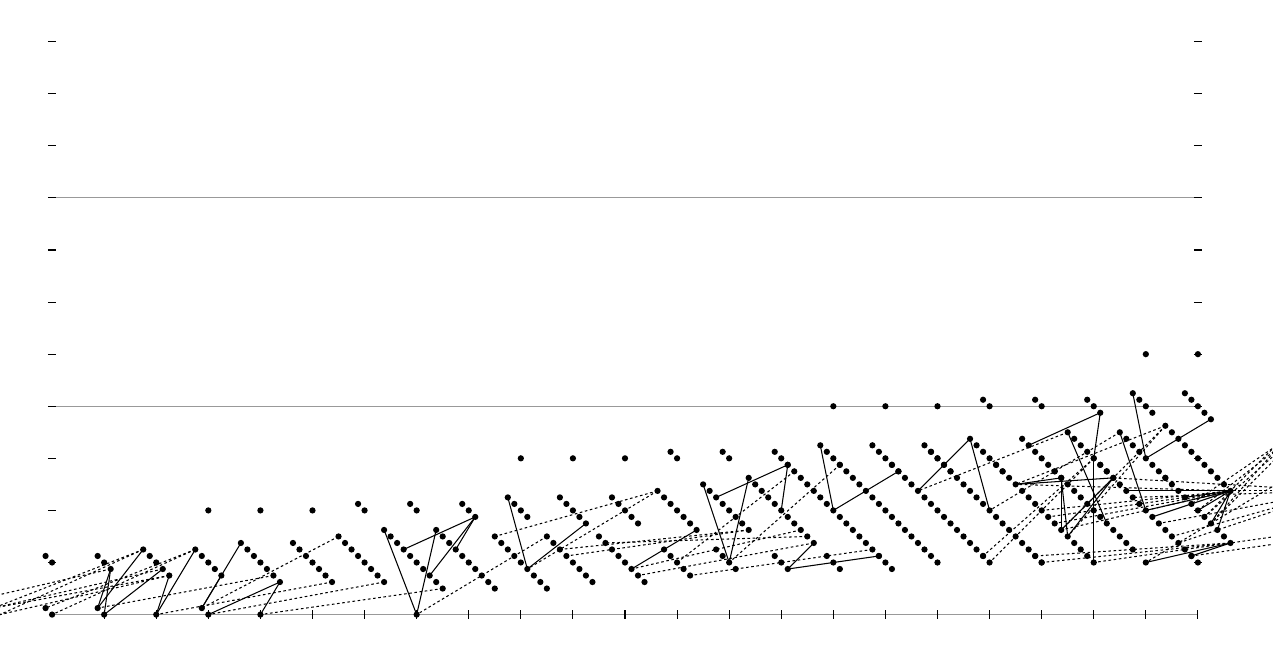}
 \caption{The file \texttt{ZDZ\_1.pdf} displaying $Ext_{A}^{s,t}(B_2, B_2)$ for \\ $-7 \leq t-s \leq 15.$}
\label{ZDZ_1} 
\end{figure}

\end{document}